\documentclass[11pt]{amsart}

\usepackage{amssymb,amsthm,amsmath}
\usepackage[numbers,sort&compress]{natbib}
\usepackage{color}
\usepackage{graphicx}
\usepackage{tikz}
\usepackage[colorlinks,linkcolor=blue,anchorcolor=blue,citecolor=blue]{hyperref}

\newtheorem{theorem}{Theorem}[section]
\newtheorem{lemma}[theorem]{Lemma}
\newtheorem{proposition}[theorem]{Proposition}
\newtheorem{corollary}[theorem]{Corollary}
\newtheorem{remark}[theorem]{Remark}

\newtheorem{definition}[theorem]{Definition}

\newcommand{\C}{{\mathbb C}}

\newcommand{\N}{{\mathbb N}}
\newcommand{\Q}{{\mathbb Q}}
\newcommand{\R}{{\mathbb R}}
\newcommand{\T}{{\mathbb T}}

\newcommand{\Z}{{\mathbb Z}}

\catcode`\@=12

\def\Empty{}
\newcommand\oplabel[1]{
  \def\OpArg{#1} \ifx \OpArg\Empty {} \else
    \label{#1}
  \fi}
\newcommand{\comm}[1]{}
\newcommand{\comment}[1]{}

\begin{document}
\title[Mobility edge]{Exact mobility edges for 1D quasiperiodic models}

\author{Yongjian Wang}
\address{School of Mathematical Sciences, Laboratory of Mathematics and Complex Systems, MOE, Beijing Normal University, 100875 Beijing, China} \email{wangyongjian@amss.ac.cn}

\author{Xu Xia}
\address{Chern Institute of Mathematics and LPMC, Nankai University, Tianjin 300071, China}
 \email{xiaxu14@mails.ucas.ac.cn}

\author{Jiangong You}
\address{
Chern Institute of Mathematics and LPMC, Nankai University, Tianjin 300071, China} \email{jyou@nankai.edu.cn}

\author {ZuoHuan  Zheng}
\address{
University  of Chinese  Academy of Sciences, Beijing 100049,China\&Academy of Mathematics and Systems Science,  Chinese Academy of Sciences,  Beijing  100190, China\&College of Mathematics and Statistics,   Hainan Normal University, Haikou, Hainan 571158, China.
} \email{zhzheng@amt.ac.cn}

\author{Qi Zhou}
\address{Chern Institute of Mathematics and LPMC, Nankai University, Tianjin 300071, China}
\email{qizhou@nankai.edu.cn}

\date{\today}

\begin{abstract}

Mobility edges (ME), i.e. critical energies which separate absolutely continuous spectrum and purely point spectrum, is an important issue in quantum physics. So far there are two experimentally feasible 1D quasiperiodic models that have been discovered to have exact mobility edge. However, all the theoretical studies have remained at the numerical level. In this paper, we rigorously prove the existence and give the precise location of the MEs for these models. 
\end{abstract}

\setcounter{tocdepth}{1}

\maketitle

\section{Introduction}

In his  1958 seminal article \cite{Anderson}, Anderson argued that in one-dimensional  or two-dimensional disordered systems, all states are localized  at any disorder strengths. However, in a three-dimensional disordered system,  a transition occurs at a finite disorder strength,  i.e., there exists a critical energy  $E_c$ separating the localized states and the extended states. This kind of   
phenomenon became known as the  \textit{Anderson metal–insulator transition}, and the critical energy $E_c$ was later termed the \textit{mobility edge} (ME) by Mott. 
The idea of mobility edges would develop into one of the most studied concepts of condensed-matter physics.
It  has  been the progenitor of many important problems in physics \cite{EM}, and was one of the main reasons  why Anderson and Mott shared the 1977 Nobel Prize in Physics.

The standard mathematical interpretation of Anderson transition is the following:  the $d$-dimensional ($d\geq 3$) random Schr\"odinger operator 
$$H= - \Delta+ V,$$ where  $V(n)$ is an independent identically
distributed  random variable with distribution uniformly in $(-\lambda,\lambda)$,  
  has Anderson localization (pure point spectrum with exponentially decaying eigenfunctions) in the regime $\pm[E_c, 2d+\lambda]$, and
absolutely continuous spectrum 
in the interval $[-E_c,E_c]$ for some $E_c$, if $\lambda$ is small. 

Over 40 years after Anderson-Mott’s Nobel Prize and  60 years after Anderson first proposed the theory,   great progress has been made in understanding the corresponding physics, however  experimental  demonstration was  notoriously difficult due to the problems in reliably controlling disorder in solid-state systems \cite{BD,EM}.
On the other hand,  the mathematical understanding of the whole picture is still unsatisfactory and one-sided: we know that if the coupling constant $\lambda$ is large enough, the corresponding Schr\"odinger operator has Anderson localization \cite{AM,FS,GMP}. But up to now,  there are no rigorous results on the existence of the absolutely continuous spectrum   for any random operators, not to  mention the existence of ME. Indeed, this is such an important question that Simon \cite{S1} gave it as Problem 1 of a list of Schr\"odinger operator problems for the twenty-first century. One can consult \cite{GK} and the references therein for recent study on this subject.

The breakthrough came in the manipulation of ultra-cold atoms, which offer a completely new, well-controlled tool for directly observing ME \cite{B,Ro}. 
Consequently there is growing interest in exploring ME in 1D quasi-periodic models, especially exact ME to understand the extended-localized transition and to advance in-depth study of fundamental ME physics, e.g. to possibly eliminate the theoretical dispute on whether many-body MEs exist~\cite{Roeck2016,Gao2019}.  However,
finding experimentally  realistic  1D quasi-periodic models with exact ME is difficult, and so far there are only  two models in  physics literature \cite{GPD,WXYZ}.
  In this paper, we rigorously  prove  ME for these two models.

Before introducing the models and our main results, let us first revisit the spectral results of the almost Mathieu operator (resp. Aubry-Andre model in physics literature):
\begin{equation*}
(H_{\lambda,\alpha,\theta} u)_n= u_{n+1}+u_{n-1} +2\lambda \cos 2
\pi (n\alpha + \theta) u_n,
\end{equation*}
where $\theta\in \mathbb{R}$ is  the phase, $\alpha\in \R\backslash
\Q$ is the frequency,  and $\lambda\in \R$ is  the coupling constant.
The almost Mathieu operator (AMO) is the central quasi-periodic model,  not only because of  its importance
in  physics \cite{AOS},
but also as a fascinating mathematical object. It was first introduced by Peierls \cite{Pe},
as a model for an electron on a 2D lattice, acted on by a homogeneous
magnetic field \cite{Ha}, and it plays a central role in the Thouless et al.
theory of the integer quantum Hall effect \cite{TKNN}.  We recall that  $\alpha$ is Diophantine (denoted by $DC( \gamma, \sigma)$), if  there exist $ \gamma, \sigma>0$, such that  
$$\|k\alpha\|_{\R/\Z} \geq \frac{\gamma}{|k|^{\sigma}} \footnote{Here we denote $\|x\|_{\R/\Z} = \inf{p\in\Z} |x-p|$} \quad  \forall k\neq 0.$$ We also denote  $DC=\cup_{\gamma>0,\sigma>0}DC( \gamma, \sigma)$. It is well known that if $\alpha\in DC$, then  $\lambda=1$ is the transition line from absolutely continuous spectrum to Anderson localization  \cite{A01,AJ1,jitomirskaya1999metal}. However,  one should note that if $\alpha$ is not Diophantine, then there exists a second transition line from singular continuous spectrum to Anderson localization \cite{AJZ,AYZ,JLiu}, which is neglected in the physics references. In any case, one has found that ME does not exist for AMO; nevertheless, based on a vast body of numerical work (one may consult \cite{BD,GPD,HK1989,Xie1988} and the references therein), the physical intuition is that if the symmetry of the almost Mathieu operator is broken in some controlled way, then the transition point $\lambda=1$ modifies into a ME. Different from random models, ME of quasiperiodic models could be any point in an interval due the the existence of gaps \cite{Kohmoto1989}.

\subsection{ME for the Generalized Aubry-Andre model}

Our first result concerns the Generalized Aubry-Andre (GAA) model:
\begin{equation}\label{model_1}
(H_{V_{1},\alpha,\theta}u)_n=u_{n+1}+u_{n-1}+2\lambda\frac{\cos2\pi (\theta+n\alpha)}{1-\tau \cos2\pi(\theta+n\alpha)}u_{n},
\end{equation}
where $\tau\in(-1,1)$.  If $\tau=0$, it is exactly  AMO, and in the limiting case $\tau=-1$,  it is the unbounded operator with potential $\tan^2(\pi \theta)$. 
 This model was first introduced by Ganeshan-Pixley-Das Sarma \cite{GPD}, where they not only give numerical evidence, but also introduce a generalized duality symmetry and show that 
\begin{equation}\label{me-1}sgn(\lambda)\tau E=2(1-|\lambda|)\end{equation}
is the ME. However,  one should note that the generalized duality is mathematically rigorous. In this paper, without invoking   generalized duality, we rigorously show that \eqref{me-1}
really defines the exact ME.  

It is well known that for any almost-periodic Schr\"odinger operator with potential $V$, its spectrum $\Sigma(V)$ is a perfect and compact set independent of the phase $\theta$. Denote 
$$ \underline{E}(V) =\min_{E} \Sigma(V) ,\qquad   \overline{E}(V) = \max_{E} \Sigma(V),$$
then our precise result can be formulated as follows:

\begin{theorem}\label{thm-gaa}
For any $\alpha\in DC$, $|\tau|<1$,  $\lambda\tau>0$, we have the following:
\begin{enumerate}
\item If $|\lambda|<1-\frac{|\tau|}{2}\overline{E}(V_{1})$, then $H_{V_1,\alpha,\theta}$ has purely absolutely continuous spectrum
for every $\theta.$
\item If $|\lambda|>1-\frac{|\tau|}{2}\underline{E}(V_{1})$, then  $H_{V_1,\alpha,\theta}$ has Anderson localization for almost  every $\theta.$
\item If $1-\frac{|\tau|}{2}\overline{E}(V_{1})<|\lambda|<1-\frac{|\tau|}{2}\underline{E}(V_{1})$, then $sgn(\lambda)\tau E=2(1-|\lambda|)$ is the ME. More precisely, 
\begin{itemize}
\item  $H_{V_1,\alpha,\theta}$ has purely absolutely continuous spectrum for every $\theta$ in  the set $\{E: sgn(\lambda)\tau E<2(1-|\lambda|)\}$.
\item $H_{V_1,\alpha,\theta}$ has Anderson localization for almost every $\theta $ in the set  $\{E: sgn(\lambda)\tau E>2(1-|\lambda|)\}$.\end{itemize}
\end{enumerate}
\end{theorem}   

\begin{remark}
One can consult Corollary \ref{theorem3.9-1} for the case $\lambda\tau<0$.
\end{remark}

\begin{figure}[h]
  \centering
  \includegraphics[width=8cm]{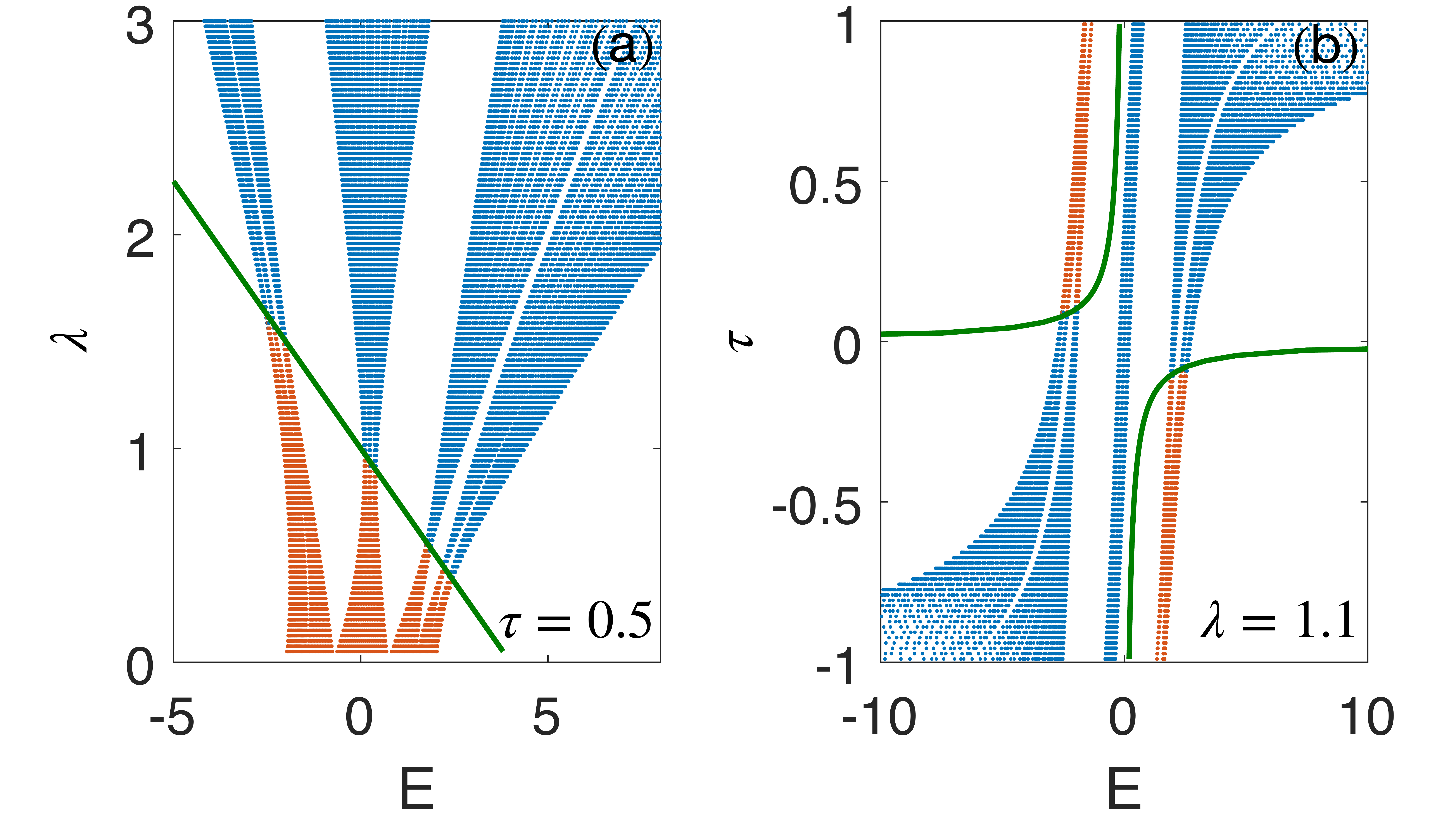} \\
  \caption{ME of the GAA model.}
\end{figure}

 Figure 1  gives a numerical picture of ME of the GAA model, where orange color corresponds to  extended states, and blue color corresponds to localized states. 
The physical mechanism of ME was explained by Anderson-Mott; we highlight that  using synthetic lattices of laser-coupled atomic momentum modes \cite{G2021}, the GAA model can be  experimentally realized to host the exact ME defined by \eqref{me-1}.  Theorem \ref{thm-gaa}  gives the rigious proof of the existence of ME \footnote{Theorem \ref{thm-gaa} covers partial result of our preprint \cite{WZ}, which is not intended for publication.}, and now the picture of ME from physics to mathematics is complete.

\subsection{ME for quasi-periodic Mosaic model}

Recently,  the following  quasi-periodic mosaic model  was proposed in \cite{WXYZ}
\begin{equation}\label{model_2}
	(H_{V_{2},\alpha,\theta}u)_n=u_{n+1}+u_{n-1}+V_{\theta}(n)u_{n},
\end{equation}
where
\begin{equation*}\quad
	V_{\theta}(n)=\left\{\begin{matrix}2\lambda \cos2\pi\theta,&n\in \kappa \mathbb{Z},\\ 0,&else,\end{matrix}\right.\quad\lambda>0.
\end{equation*}
This model certainly defines a family of almost-periodic Schr\"odinger operators. If $\kappa=1$, then one can reduce it to AMO. As pointed out in \cite{WXYZ}, the model is  experimentally realizable using  an optical Raman lattice, thus a true physical model.  
We will show that, different from AMO,  the mosaic model (\ref{model_2}) with $\kappa\ge 2$ do have  MEs and we also give exact and complete description of all  mobility edges for $\kappa= 2, 3$.

\begin{theorem}\label{thm-mosaic}
Let $\lambda \neq 0$,  $\alpha\in DC$.
If $\kappa=2$, then we have the following:
\begin{enumerate}
\item If $|\lambda| \overline{E}(V_2)< 1$, then $H_{V_2,\alpha,\theta}$ has purely absolutely continuous
spectrum
for every $\theta.$
\item If $|\lambda| \overline{E}(V_2) > 1$, then $\pm \frac{1}{\lambda}$ are MEs. More precisely, 
\begin{itemize}
\item  $H_{V_2,\alpha,\theta}$ has purely absolutely continuous spectrum  in $\Sigma(V_2)\cap (-\frac{1}{\lambda}, \frac{1}{\lambda})$  for every $\theta$.
\item  $H_{V_2,\alpha,\theta}$ has Anderson localization  in $\Sigma(V_2)\cap [-\frac{1}{\lambda}, \frac{1}{\lambda}]^c$  for almost every $\theta$.
\end{itemize}
\end{enumerate}
\end{theorem}   

\begin{figure}[h]
  \centering
  \includegraphics[width=10cm]{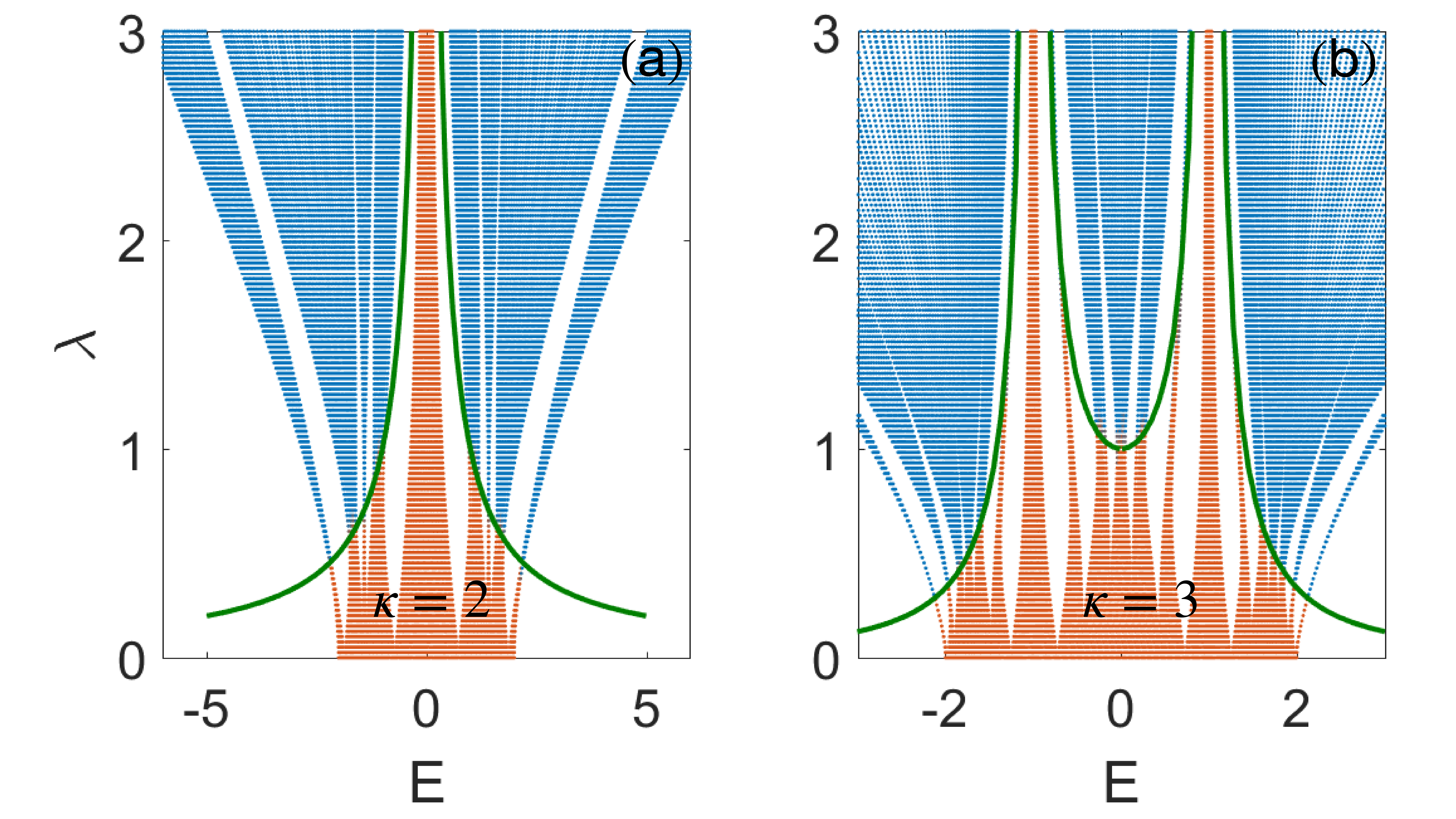} \\
  \caption{ME of the quasi-periodic mosaic model.}
\end{figure}

 Figure 2  gives a numerical picture of ME of the quasi-periodic mosaic model. As is clear from the picture,  the localization starts from the edges of the spectrum, and as the coupling constant $\lambda$ is increased, then we have mobility edges, which move towards the center of the spectrum. This kind of behavior is similar to that of 3D disordered systems \cite{LR}. However, our results really demonstrate a new phenomenon, which does not even appear in previous physics literature. That is, no matter how large the coupling constant is, ME always occur. By contrast, in  the random models or the quasi-periodic models (with smooth potential), all the states are believed to be localized  when $\lambda$ is large enough \cite{AM,b1,b2,BG,bgs2,FS,GMP}.

Also from  Figure 2(a),  if  $\kappa=2$,  it is clear \eqref{model_2} has two mobility edges. In general, one can anticipate arbitrary many even numbers of ME (Figure 2(b) for $\kappa=3$). In case $\kappa=3$, and denote
$$E_{c}^1=   \sqrt{1+ \frac{1}{\lambda}}, \qquad   E_{c}^2=   \sqrt{1- \frac{1}{\lambda}},$$
 then the complete picture is the  following:

\begin{theorem}\label{thm-mosaic-k3}
Let $\lambda \neq 0$,  $\alpha\in DC$.
If $\kappa=3$, then we have the following:
\begin{enumerate}
\item If $|\lambda| (\overline{E}(V_2)^2-1)< 1$, then $H_{V_2,\alpha,\theta}$ has purely absolutely continuous
spectrum
for every $\theta.$
\item If  $ \frac{1}{ \overline{E}(V_2)^2-1 } < |\lambda| < 1$, then  $\pm E_{c}^1$ are MEs. More precisely, 
\begin{itemize}
\item  $H_{V_2,\alpha,\theta}$ has purely absolutely continuous spectrum  in $\Sigma(V_2)\cap (- E_{c}^1, E_{c}^1)$  for every $\theta$.
\item  $H_{V_2,\alpha,\theta}$ has Anderson localization  in $\Sigma(V_2)\cap [- E_{c}^1, E_{c}^1]^c$  for almost every $\theta$.
\end{itemize}
\item If  $ |\lambda| > 1$,  then   $\pm E_{c}^1$, $\pm E_{c}^2$ are MEs. More precisely, 
\begin{itemize}
\item  $H_{V_2,\alpha,\theta}$ has purely absolutely continuous spectrum  in $\Sigma(V_2)\cap ( E_c^2, E_c^1)$ and  $\Sigma(V_2)\cap (-E_c^1, E_c^2)$  for every $\theta$.
\item  $H_{V_2,\alpha,\theta}$ has Anderson localization  in $\Sigma(V_2)\cap [- E_c^1,   E_c^1]^c$  and  $\Sigma(V_2)\cap (- E_c^2, E_c^2)$ for almost every $\theta$.
\end{itemize}
\end{enumerate}
\end{theorem}   

\begin{remark}
One can consult the result for general $\kappa$ in Theorem \ref{thm-mosaic-2}. 
\end{remark}

\subsection{Other models}

The third model concerns the tight-binding model
\begin{equation}\label{model_3}
(\widehat{H}_{V_{3},\alpha,\theta}x)_{n}=\sum\limits_{j\neq n}e^{-p|n-j|}x_{j}+\lambda \cos2\pi(n\alpha+\theta)x_{n},\end{equation}
with parameter $p>0$. This is a quasi-periodic long-range operator acting on $\ell^{2}(\Z)$. This quasi-periodic model was  introduced by Biddle-Das Sarma in their groundbreaking work \cite{BD},  
where they predicted  
$$E+1=\cosh(p) |\lambda| $$
is the exact energy dependent mobility edge, and this gives the first model which has exact ME in the physics literature.  In this paper, we will actually show that the Aubry dual of \eqref{model_3} reduces to the GAA model, and as a consequence, we will rigorously show the following:

\begin{corollary}\label{tb}
For any $\lambda \neq 0$,  $\alpha\in DC$, the ME of  $\widehat{H}_{V_{3},\alpha,\theta}$ takes place at 
$$E+1=\cosh(p) |\lambda| .$$
\end{corollary}

\begin{remark}
One can consult the precise result at Corollary \ref{corollary3.11}.
\end{remark}

 The  final  model is the Schr\"odinger operator with ``Peaky" potential 
\begin{equation}\label{model_4}
	(H_{V_{4},\alpha,\theta}u)_n=u_{n+1}+u_{n-1}+\frac{\lambda}{1+4 K \sin^2\pi( \theta+n\alpha) }u_{n},\quad K,\lambda>0,
\end{equation}
which was first introduced by Bjerkl\"ov and Krikorian \cite{BK}. Theorem B of \cite{BK} shows that for some sufficiently large  $K$ and $\lambda$,  there is a set $\mathcal{A}\subset \T$ of positive Lebesgue measure such that for any $\alpha\in \mathcal{A}$ the operator $H_{V_4,\alpha, \theta}$ has both a.c. and p.p. components. In this paper, we will reveal  exactly when this operator has ME and where is the ME.

\begin{figure}[h]
	\centering
	\includegraphics[width=10cm]{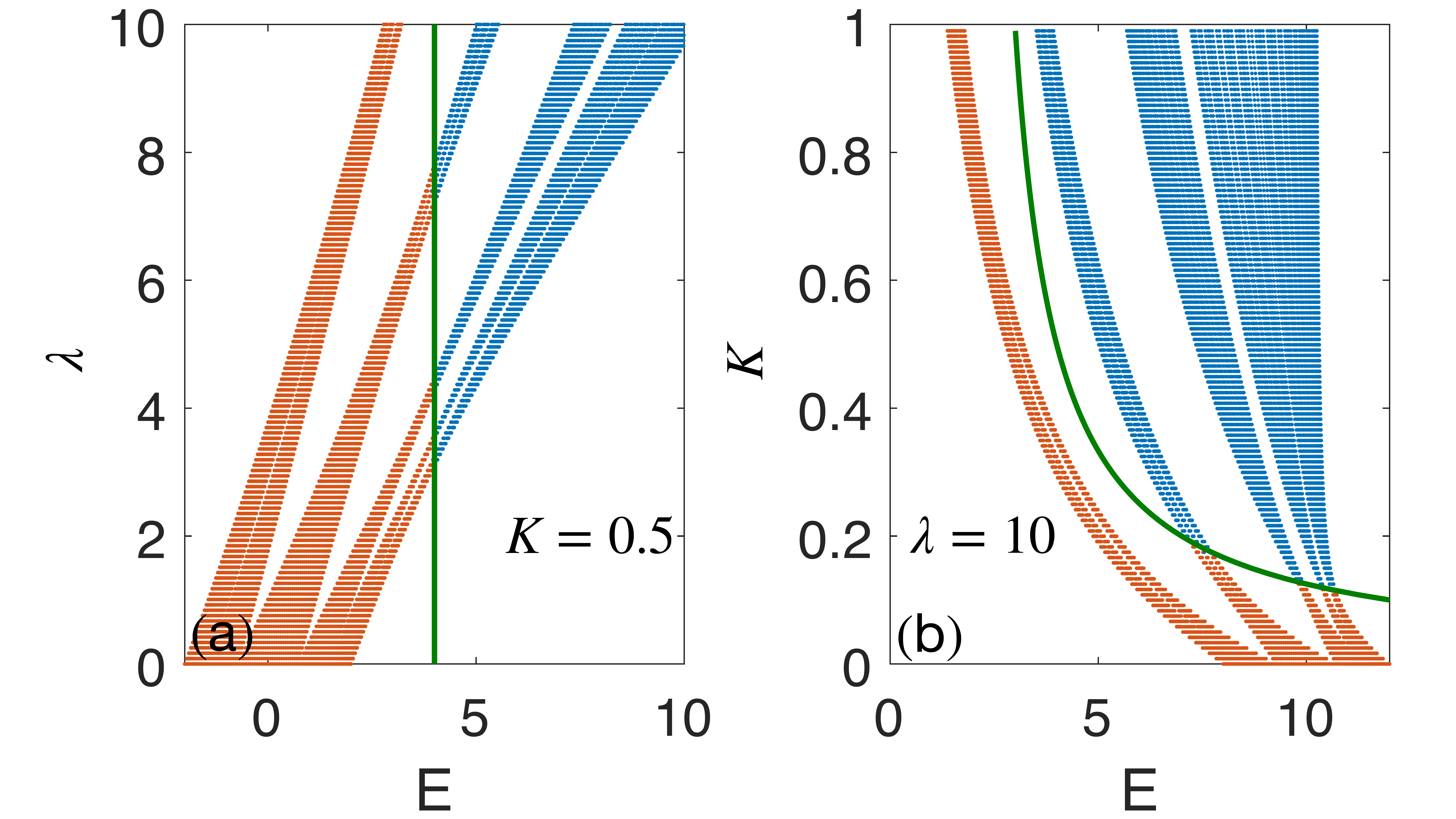} \\
	\caption{ME of Schr\"odinger operator with ``Peaky" potential .}
\end{figure}

\begin{corollary}\label{model_peaky}
Let  $K>0$ and $\alpha\in DC$. Then  the following holds true: 
	\begin{enumerate}
		\item If $  \frac{\lambda K}{(2K+1)^2} < 1- \frac{K  \overline{E}(V_{4}) }{2K+1}$, then $H_{V_4,\alpha,\theta}$ has purely absolutely continuous for every $\theta.$
		\item If $ \frac{\lambda K}{(2K+1)^2} > 1- \frac{\lambda \underline{E}(V_{4}) }{2K+1} $, then   $H_{V_4,\alpha,\theta}$ has Anderson localization for almost  every $\theta.$
		\item  If $  1- \frac{\lambda  \overline{E}(V_{4}) }{2K+1}  <  \frac{\lambda K}{(2K+1)^2} <  1- \frac{K \underline{E}(V_{4}) }{2K+1} $, then  $2+\frac{1}{K}$ is the ME. More precisely, 
		\begin{itemize}
		\item  $H_{V_4,\alpha,\theta}$ has purely absolutely continuous spectrum  in  $\Sigma(V_4)\cap[\underline{E}(V_4),  2+\frac{1}{K})$    for every $\theta.$
		\item  $H_{V_4,\alpha,\theta}$ has Anderson localization in 	$\Sigma(V_4)\cap(2+\frac{1}{K}, \overline{E}(V_4)]$ for almost  every $\theta.$
		\end{itemize}
		\end{enumerate}
		\end{corollary}	

\begin{remark}
We point out an interesting phenomenon, as is also clearly shown in  Fig 3,    that    the  ME of \eqref{model_4} doesn't depend on  the coupling constant $\lambda$. So one sees that  $K$ gives the location of ME while $\lambda$ determines whether ME will appear.
\end{remark}

\subsection{Coexistence of spectrums.} 

The coexistence of  a.c. and p.p. spectrum is an active research subject which is related to and  obviously weaker than exact ME. 
Bjerkl\"ov  \cite{Bj} proved that if the potential is
$$V(\theta)= \exp(Kf(\theta+\alpha))+ \exp(-Kf(\theta)),$$
where $f$ is assumed to be a non-constant real-analytic function with zero mean,
then the Schr\"odinger operator has coexistence of regions of the spectrum with positive Lyapunov exponents and zero Lyapunov exponents, if $K$ is large enough. 
Elaborating on \cite{Bj}, Zhang \cite{zhang} gives examples of the coexistence of a.c. and p.p. spectrum and coexistence of a.c. and s.c.
 spectrum. Bjerl\"ov and Krikorian \cite{BK} constructed a class of ``peaky" potentials, such that the operator has  coexistence of a.c. and p.p. spectrum.  Avila \cite{A4} constructed examples of potentials which are real analytic perturbations of critical AMO, and for which the spectrum of the corresponding Schr\"odinger operator has both a.c. and p.p. components.   For previous coexistence results on quasi-periodic potentials with two frequencies and almost periodic potential, one can consult \cite{bourgain,FK}.    

 An effective method for proving coexistence of spectrum is studying the Lyapunov exponent of the   {\it Schr\"odinger cocycles}\footnote{One can consult Section \ref{co-ly} for its definition} associated to Schr\"odinger operators. This is a family of skew-products
$$(\alpha, S_{E}^V): \T \times \R^2  \circlearrowleft, \qquad   (\alpha, S_{E}^V) (\theta, v) = ( \theta+\alpha,S_{E}^V(\theta) \cdot v)$$
where \begin{equation*}
	S_{E}^{V}(\cdot)=\begin{pmatrix}E-V(\cdot)&-1\\1&0\end{pmatrix}.
\end{equation*}
More precisely, coexistence of zero and positive Lyapunov exponents in the spectrum roughly implies the coexistence of a.c. and p.p. spectrum, since by the well-known Kotani's theory \cite{kotani},    $\Sigma_{ac}(V)$ is the essential support of the energies which have zero Lyapunov exponent, and it is a commonly used fact in physics literature that positive Lyapunov exponent implies localization.  

\subsection{Main ingredients of the proof}

 We stress that the above mentioned results  \cite{Bj,BK,zhang} only give coexistence results, i.e. partial information on the spectrum, while ME requires {\it complete} information on the spectrum. For this purpose, we need to use the remarkable  global theory of one-frequency analytic cocycles by Avila   \cite{A4}, where he establishes  and gives classification of all $SL(2,\C)$ cocycles. To be precise, cocycles that are not uniformly hyperbolic are classified in three regimes:
\begin{enumerate}
\item \textit{Subcritical}, if there exists $\delta>0$ such that $L(\alpha, A(z))=0$  through some strip  $|\Im z|\leq \delta$,
\item \textit{Supercritical}, or nonuniformly hyperbolic, if $L(\alpha, A)>0$,
\item  \textit{Critical} otherwise. 
\end{enumerate}
In the subcritical regime, the energy is related
with extended states, while in the supercritical regime, the energy is related with localized states.  To study ME, there are three key steps. The first is to find out the exact formula of  the  Lyapunov exponent $L(\alpha, S_{E}^V)$ in the spectrum, which allows one  to locate the zero Lyapunov exponent regime and positive Lyapunov exponent regime. Then one needs to prove a.c. spectrum in the subcritical regime, and prove localization in the supercritical regime. 
\\

\noindent \textbf{Calculation of Lyapunov exponents.} As we said,  to obtain exact ME,  the first step is to calculate the Lyapunov exponent.  Based on the continuity of the Lyapunov exponent \cite{BJ} and the Lyapunov exponent in the rational frequencies \cite{Kra}, Bourgain and Jitomirskaya \cite{BJ} showed that if the energy belongs to the spectrum, then the Lyapunov exponent  of AMO satisfies 
\begin{equation}\label{le-formula-amo}L(\alpha,S_E^{2\lambda \cos})=\max
\{0,\ln |\lambda|\}.\end{equation}
However, this method can hardly be generalized. On the other hand,  Avila's global theory shows that, as a function of $\epsilon,$  the Lyapunov exponent $L(\alpha,S_E^{V}(\cdot+ i\epsilon))$ is a convex, piecewise linear function, with integer slopes. Based on this fact,  Avila \cite{A4} gives another proof of \eqref{le-formula-amo}. In this paper, we will further generalize 
this argument, and calculate the Lyapunov exponent of  the GAA model (Lemma \ref{lemma3.7}) and quasi-periodic mosaic model (Lemma \ref{lemma3.11}), and more importantly locate the subcritical and supercritical regime.   Note that this method strongly depends on the fact that the acceleration of the Lyapunov exponent\footnote{Consult section \ref{acceleration} for its definition} (the slope of the Lyapunov exponent) is not larger than $1$, and this also explains why it is so difficult to find models with exact ME. \\

\noindent \textbf{Absolutely continuous spectrum.} 
Based on the KAM method,  Dinaburg-Sinai \cite{DS} proved that if $\alpha \in DC$, then $\Sigma_{ac}(\lambda V) \neq \varnothing$ in the \textit{perturbative small regime $\lambda<\lambda_0$}.
Here perturbative means that $\lambda_0$ depends on $\alpha$ through the Diophantine constants $\gamma,\sigma$.  Under the same assumption,  Eliasson \cite{E92} showed that in fact the spectrum is  purely absolutely continuous for any $\theta$.  Specifically in the one-frequency case, one can even anticipate non-perturbative results.  Making use of the specificity of one frequency,  some new elaborate   techniques have been developed
to prove some sharp results.  If $\alpha\in DC$, based on non-perturbative Anderson localization results, Avila-Jitomirskaya \cite{AJ1} proved that there exists $\lambda_1$ which  does not depend on $\alpha$, such that  $\Sigma(\lambda V)=\Sigma_{ac}(\lambda V)$ when $\lambda< \lambda_1$. Such a result was generalized  by Avila to the weak Diophantine case \cite{A01}. 
Recently,  Avila-Fayad-Krikorian \cite{AFK} and Hou-You \cite{HY} independently  developed  non-standard KAM
techniques, and showed that  $\Sigma_{ac}(\lambda V) \neq \varnothing$ for  $\lambda< \lambda_2(V)$ and for any  irrational $\alpha$. The breakthrough goes back to Avila, who established the deep relations between the existence of a.c. spectrum  and the vanishing of the Lyapunov exponent. To be precise, his  \textit{Almost Reducibility Conjecture} (ARC) says that any  \textit{subcritical cocycle} is almost reducible, which furthermore supports a.c. spectrum. Our proof relies on the solution of ARC, as announced in \cite{A4}, to appear in \cite{avila2010almost,avilalyapunov}. ARC has many important dynamical and spectral consequences \cite{avila2010almost,avilalyapunov,AJM,AKL,AYZ, GY,LYZZ}, indeed, it was already stated as  \textit{Almost Reducibility Theorem} (ART) in \cite{AJM}.

In our case, for the GAA model, one only needs to locate the subcritical regime, then one applies ARC directly to prove that the corresponding regime has pure a.c. spectrum. However, for the quasi-periodic mosaic model, the operator itself cannot  induce a  quasi-periodic Schr\"odinger cocycle. The observation here is that the iterates of the cocycle can be seen as a one-frequency analytic cocycle,  thus one can locate the subcritical regime by Avila's global theory, however ARC cannot apply directly, since an iterate of the cocycle
does not define an operator any more. 
Here, we will develop a scheme to establish the link between absolutely continuous spectrum  of almost periodic operators and almost reducibility  of its iterated cocycle; the ideas first goes back to  Avila \cite{A01}, while the estimates are KAM based \cite{CCYZ,LYZZ}. One can found more discussions after Theorem \ref{pac} the difficulty and necessity for us to develop a general scheme. Indeed, such a scheme has already been used to study the purely a.c. spectrum of CMV matrices with small quasi-periodic Verblunsky coefficients \cite{LDZ}.   \\

\noindent  \textbf{Anderson localization.} 
The above mentioned coexistence papers  \cite{A4, BK,zhang} all depend  crucially on Bourgain-Goldstein's result \cite{BG}, where  they prove that in the supercritical regime, for any {\it fixed} phase, $H_{\lambda V,\alpha,\theta}$ has AL for $a.e.$ Diophantine frequency, i.e  they have to remove a Hausdorff zero measure set of Diophantine frequencies. For the multi-frequency and multi-dimensional case, one can consult \cite{b1,b2,bgs2, jls,JSY,K} and the references therein.  However, in physics applications, there is more interest in the case where $\alpha$ is a priori fixed as a Diophantine frequency. 
For localization results with fixed Diophantine frequency,  if the potential is a cosine-like function,  Fr\"ohlich-Spencer-Wittwer \cite{fsw} and Sinai \cite{Sinai} independently proved that  for a.e.  phase,  $H_{\lambda V,\alpha,\theta}$   has AL for sufficiently large coupling constant. If the potential is analytic,  Eliasson \cite{Eli97} proved that  $H_{\lambda,\alpha,\theta}$ has pure point spectrum for a.e. $\theta$ and large enough $\lambda.$ 

One can see that although these three localization results  \cite{Eli97,fsw,Sinai} hold for fixed Diophantine frequency, they are all perturbative, i.e. the coupling constant $\lambda$ is assumed   to be large enough.   
It is still open whether for non-constant analytic potentials and fixed Diophantine frequency, the operator $H_{V,\alpha,\theta}$  has Anderson localization for a.e. $\theta$ in the supercritical regime. To this stage, we should mention Jitomirskaya's seminar paper  \cite{jitomirskaya1999metal}, who not only  proves  Anderson localization result for the almost Mathieu operator, but also developed a non-perturbative localization approach which initiated other non-perturbative localization results (one may consult \cite{AJ1,b1,b2,BJ02,BG,jks,JLiu,jls,jy} and the reference therein). 
In this paper, we will further develop Jitomirskaya's argument, and show that AL still holds for another family of analytic quasi-periodic Schr\"odinger operator in the whole supercritical regime.

\section{Preliminaries}
 For a bounded analytic function $f$ defined on a strip $\{|\Im z|<h\}$, let $\mathop{||f||}_{h}=\sup_{|\Im\theta<h|}||f(\theta)||$ and denote by $C_{h}^{\omega}(\mathbb{T},*)$ the set of all these $*$-valued functions ($*$ will usually denote $\mathbb{R}$, $SL(2,\mathbb{R})$, $M(2,\mathbb{C})$). When $\theta\in\mathbb{R}$, we also set $||\theta||_{\mathbb{T}}=\inf_{j\in\mathbb{Z}}|\theta-j|$.

\subsection{Continued Fraction Expansion.}\quad Let $\alpha\in(0,1)$ be irrational, $a_{0}=0$ and $b_{0}=\alpha$. Inductively, for $k\ge1$, we define
\begin{equation*}\quad
a_{k}=\lfloor b_{k-1}^{-1}\rfloor,\ b_{k}=b_{k-1}^{-1}-a_{k},
\end{equation*}
Let $p_{0}=0$, $p_{1}=1$, $q_{0}=1$, $q_{1}=a_{1}$. We define inductively $p_{k}=a_{k}p_{k-1}+p_{k-1}$, $q_{k}=a_{k}q_{k-1}+q_{k-2}$. Then $(q_{n})_{n}$ is the sequence of denominators of the best rational approximations of $\alpha$, since we have $||k\alpha||_{\mathbb{T}}\ge||q_{n-1}\alpha||_{\mathbb{T}}$, $\forall\ 1\le k<q_{n}$, and 
\begin{equation*}\quad
\frac{1}{2q_{n+1}}\le||q_{n}\alpha||_{\mathbb{T}}\le\frac{1}{q_{n+1}}.
\end{equation*}

\begin{lemma}\cite{AJ}\label{ten}
	Let $\alpha\in\mathbb{R}\backslash\mathbb{Q}$, $x\in\mathbb{R}$ and $0\le l_0\le q_n-1$ be such that $$|\sin\pi(x+l_0\alpha)|=\inf_{0\le l\le q_n-1}|\sin\pi(x+l\alpha)|,$$ then for some absolute constant $C>0$,
	$$-C\ln q_n\le\sum_{0\le l\le q_n-1,l\neq l_0}\ln|\sin\pi(x+l\alpha)|+(q_n-1)\ln2\le C\ln q_n.$$
\end{lemma}

\subsection{Cocycle, Lyapunov exponent} \label{co-ly}
Let $X$ be a compact metric space, $(X, \nu, T)$ be ergodic. A cocycle $(\alpha, A)\in \R\backslash
\Q\times C^\omega(X, M(2,\R))$ is a linear skew product:
\begin{eqnarray*}\label{cocycle}
(T,A):&X \times \R^2 \to X \times \R^2\\
\nonumber &(x,\phi) \mapsto ( T x,A(x) \cdot \phi).
\end{eqnarray*}
For $n\in\mathbb{Z}$, $A_n$ is defined by $(T,A)^n=(T^n,A_n)$. Thus $A_{0}(x)=id$,
\begin{equation*}
A_{n}(x)=\prod_{j=n-1}^{0}A(T^{j}x)=A(T^{n-1}x)\cdots A(Tx)A(x),\ for\ n\ge1,
\end{equation*}
and $A_{-n}(x)=A_{n}(T^{-n}x)^{-1}$. The Lyapunov exponent is defined as
\begin{equation*}\quad
L(T,A)=\lim_{n\rightarrow\infty}\frac{1}{n}\int_{X}\ln||A_{n}(x)||dx.
\end{equation*}
In this paper, we will consider the following two useful  cocycles.
\begin{itemize}
\item  $X=\mathbb{T}$ and $T=R_\alpha$, where  $R_{\alpha}\theta= \theta+\alpha$, then $(\alpha,A):=(R_{\alpha},A)$ is a  quasi-periodic cocycle.
\item  $X=\mathbb{T}\times\mathbb{Z}_{\kappa}$ and $T=T_{\alpha}$, where $\kappa\in \Z^+$,  $T_{\alpha} (\theta,n)= (\theta+\alpha,n+1)$, then $(T_{\alpha},A)$ defines an  almost-periodic cocycle.
\end{itemize}
These dynamical  system $(X,T)$ is uniquely ergodic if  $\alpha$ is irrational (Theorem 9.1 of \cite{mane2012ergodic}).

We say an $SL(2,\mathbb{R})$ cocycle $(T,A)$  is uniformly hyperbolic if, for every $x \in X$, there exists a continuous splitting $\mathbb{R}^2=E_{s}(x)\oplus E_{u}(x)$ such that for every $n\ge0$,
\begin{equation*}\quad
\begin{split}
|A_{n}(x)v(x)|&\le Ce^{-cn}|v(x)|,\ v(x)\in E_{s}(x),\\
|A_{-n}(x)v(x)|&\le Ce^{-cn}|v(x)|,\ v(x)\in E_{u}(x),
\end{split}
\end{equation*}
for some constans $C,c>0$. Clearly, it holds that $A(x)E_{s}(x)=E_{s}(Tx)$ and $A(x)E_{u}(x)=E_{u}(Tx)$ for every $x\in X$, and if $(T,A)$ is uniformly hyperbolic, then $L(T,A)>0$.

\subsection{Fibre rotation number.}
 Let $\mathbb S^{1}$ be the set of unit vectors of $\mathbb{R}^{2}$, consider a projective cocycle $F_{A}$ on $X\times\mathbb{S}^1$: $$(x,\phi)\mapsto (Tx,\frac{A(x)\phi}{\|A(x)\phi\|}).$$
If  $A\in C^0(\T, SL(2,\R))$ is homotopic to the
identity,    then there exists
a lift $\tilde{F}_{A}$ of $F_{A}$ to $X \times \mathbb{R}$ such that $\tilde{F}_{A}(x,\phi) = (Tx, \tilde{f}_{A}(x,\phi)) $  where $\tilde{f}_{A} :X \times \mathbb{R}\rightarrow\mathbb{R}$ is a continuous lift such that
\begin{itemize}
  \item $\tilde{f}_{A}(x,\phi + 1) = \tilde{f}_{A}(x,\phi) + 1;$
  \item for every $x \in X, \tilde{f}_{A}(x, \cdot) : R \rightarrow R$ is a strictly increasing homeomorphism;
  \item if $\pi_2$ is the projection map  $X \times \mathbb{R}\rightarrow X \times \mathbb{S}^1:(x,\phi) \mapsto (x, e^{2\pi i\phi})$, then $F_{A} \circ \pi_2 = \pi_2 \circ \tilde{F}_{A}$.
\end{itemize}
If $(X, \nu, T)$ is uniquely ergodic, then the number 
\begin{equation*}
\rho(T,A) = \lim_{n\rightarrow \infty}\frac{\tilde{f}_{A}^{n}(x,\phi) -\phi}{n} \mod \Z
\end{equation*}
is independent of  $(x,\phi) \in X \times \mathbb{R/Z}$ and  the lift of $F_{A}$,  and is called the {\it fibered rotation
number} of $(T,A)$, see \cite{johnson1982rotation,herman1983methode} for details.

If $X=\mathbb{T}$ and $T=R_\alpha$, i.e. when we are dealing with quasi-periodic cocycles, we will simply denote its fiber rotation number as $\rho(\alpha, A)$.
The fibered rotation number is invariant under real conjugacies which are homotopic to the identity. In general, if the cocycles $(\alpha,A_1)$ is conjugated to $(\alpha,A_2)$: 
$$B(\theta+\alpha)^{-1}A_1(\theta)B(\theta)=A_2(\theta),$$ and $B \in
C^0(\T,$ $	PSL(2,\R))$ has degree n (that is, it is homotopic to $\theta \mapsto R_{n\theta/2}$), where 
\begin{equation*}
R_{\phi}=\begin{pmatrix}\cos2\pi\phi& -\sin2\pi\phi\\ \sin2\pi\phi& \cos2\pi\phi\end{pmatrix},
\end{equation*}
then  we have
\begin{equation}
\rho({\alpha,A_{1}})=\rho(\alpha,A_{2})+\frac{1}{2}n\alpha \mod\ \mathbb{Z}.
\end{equation}
If furthermore $B \in C^0(\T,$ $SL(2,\R))$ with  $\deg B=n \in\Z$, then we
\begin{equation*}
\rho({\alpha,A_{1}})=\rho(\alpha,A_{2})+n\alpha \mod \Z.
\end{equation*}

\subsection{Dynamical  defined Schr\"odinger operators.} 
Let $X$ be a compact metric space, $(X,\nu,T)$ be ergodic, and $V: X \rightarrow \R$ is continuous. Then one can define the Schr\"odinger operator on 
$\ell^2(\Z)$:
\begin{equation*}
(H_{V, x}u)_{n}=u_{n+1}+u_{n-1}+V(T^n x )u_{n},\ \ \forall x \in X. 
\end{equation*}
It is well known that the spectrum of $H_{V,x}$ is a compact subset of $\mathbb{R}$, independent of $x$ if $(X,T)$ is minimal \cite{damanik2017schrodinger}, we shall denote it by $\Sigma(V)$.
The integrated density of states (IDS) $N_{V}:\mathbb{R}\rightarrow[0,1]$ of $H_{V,x}$ is defined as
\begin{equation*}
N_{V}(E)=\int_{X}\mu_{V,x}(-\infty,E]d\nu,
\end{equation*}
where $\mu_{V,x}$ is the spectral measure of $H_{V,x}$. Note  any formal solution $u=(u_{n})_{n\in\mathbb{Z}}$ of $H_{V,x}u=Eu$ can be rewritten as 
\begin{equation*}
\begin{pmatrix}u_{n+1}\\u_{n}\end{pmatrix}=S_{E}^{V}( T^n x)\begin{pmatrix}u_{n}\\u_{n-1}\end{pmatrix},
\end{equation*}
where \begin{equation*}
S_{E}^{V}(\cdot)=\begin{pmatrix}E-V(\cdot)&-1\\1&0\end{pmatrix},
\end{equation*}
and we  call $(T, S_{E}^V)$  the Schr\"odinger cocycle.  It is well-known that $E\notin\Sigma(V)$ if and only if $(T,S_{E}^{V})$ is uniformly hyperbolic \cite{johnson1986exponential}.

In this paper, we are interested in the case that  $(X,T)=(\T, R_{\alpha})$ or $(\T\times \Z_{\kappa}, T_{\alpha})$, where $\alpha$ is irrational, then the base dynamics is  almost periodic (thus minimal and uniquely ergodic).  For any fixed $E\in\mathbb{R}$, the map $\theta\mapsto S_{E}^{V}(\theta)$ is homotopic to the identity, hence the rotation number $\rho(T,S_{E}^{V})$ is well defined. Moreover, $\rho(T,S_{E}^{V})\in[0,\frac{1}{2}]$ relates to the integrated density of states $N_{V}$ as follows:
$$N_{V}(E)=1-2\rho(T,S_{E}^{V}).$$
By Thouless formula \cite{AS}, we also have the following relation between the integrated density of states and the Lyapunov exponent:
$$L(T,S_{E}^{V})=\int\ln|E-E'|dN_{V}(E').$$

To get the existence of absolutely continuous spectrum, we need the following well-known result from subordinacy theory:
\begin{theorem}\cite{gilbert1987subordinacy}\label{subordinnacy}
Let $\mathcal{B}$ be the set of $E\in\R$ such that the Schr\"odinger cocycle $(T, S_{E}^V)$ is bounded. Then $\mu_{V,\theta}|\mathcal{B}$ is absolutely continuous for all $\theta\in X$.
\end{theorem}

Moreover, one can relate the growth of the cocycles to the spectral measure directly:

\begin{lemma}\cite{A01}\label{growth-cocycle}
There exists universal constant $C>0$, such that
 $\mu_{V,\theta}(E-\epsilon, E+\epsilon) \leq C\epsilon  \sup_{0\leq s\leq \epsilon^{-1}}\|(S_{E}^V)_s\|_0^2$.
\end{lemma}

\subsection{Global theory of one frequency quasiperiodic cocycle.}\label{acceleration}
Let us make a short review of  Avila's global theory of one-frequency quasi-periodic cocycles \cite{A4}.  Suppose that $D\in C^\omega(\T, M(2,\C))$ admits a holomorphic extension to $\{|\Im  \theta |<h\}$. Then for
$|\epsilon|<h$, we define $D_\epsilon \in
C^\omega(\T, M(2,\C))$ by $D_\epsilon(\cdot)=A(\cdot+i \epsilon)$, and define the
the acceleration of $(\alpha,D_{\varepsilon})$ as follows
\begin{equation*}
\omega(\alpha,D_{\varepsilon})=\frac{1}{2\pi}\lim_{h\to 0+}\frac{L(\alpha,D_{\varepsilon+h})-L(\alpha,D_{\varepsilon})}{h}.
\end{equation*}

The acceleration was first introduced by Avila  for analytic $SL(2,\C)$-cocycles \cite{A4}, and extended to analytic $M(2,\C)$ cocycles by Jitomirskaya-Marx \cite{jitomirskaya2012analytic}. It follows from the convexity and continuity of the Lyapunov exponent that the acceleration is an upper semicontinuous function in parameter $\varepsilon$. The key property of the acceleration is that it is quantized: 

\begin{theorem}[Quantization of acceleration\cite{A4,jitomirskaya2012analytic,jit2013}]\label{theorem3.4}
Suppose that  $(\alpha,D)\in(\mathbb{R}\backslash\mathbb{Q})\times C^{\omega}(\mathbb{T},M_{2}(\mathbb{C}))$ with $detD(\theta)$ bound away from 0 on the strip $\{|\Im  \theta |<h\}$, then $\omega(\alpha,D_{\varepsilon})\in \frac{1}{2}\mathbb{Z}$ in the strip. Morveover,  if $D\in C^{\omega}(\mathbb{T}, SL(2,\C))$, then 
$\omega(\alpha,D_{\varepsilon})\in\mathbb{Z}$ 
\end{theorem}

 If $A$ takes values in $SL(2,\mathbb{R})$, then $\varepsilon\mapsto L(\alpha,A_{\varepsilon})$ is an even function. By convexity, $\omega(\alpha,A)\ge0$. And if $\alpha\in\mathbb{R}\backslash\mathbb{Q}$, then $(\alpha,A)$ is uniformly hyperbolic if and only if $L(\alpha,A)>0,$ and  $\omega(\alpha,A)=0$.
The cocycles in $SL(2,\mathbb{R})$ which are not uniformly hyperbolic are classified into three regimes: subcritical, critical, and supercritical. Especially,  $(\alpha,A)$  is said to be subcritical if $L(\alpha,A)=0,$ $\omega(\alpha,A)=0$; the cocycle $(\alpha,A)$  is said to be supercritical if $L(\alpha,A)>0,$ $\omega(\alpha,A)>0$; otherwise $(\alpha,A)$  is critical. 
%
%
%

The heart of Avila's global theory is his ``Almost Reducibility Conjecture'' (ARC), which says that subcritical implies almost reducibility.  Recall that 
 a cocycle  $(\alpha,A)$ is (analytically) reducible, if it can be $C^\omega$ conjugated to a constant cocycle;  $(\alpha,A)$ is (analytically) almost reducible if the closure of its analytic conjugates contains a constant.
The full solution of ARC was recently given by Avila in \cite{avila2010almost,avilalyapunov}:
\begin{theorem}\cite{avila2010almost,avilalyapunov}\label{arc}
Given $\alpha\in\mathbb{R}\backslash\mathbb{Q}$, and $A\in C^{\omega}(\mathbb{T}$, $SL(2,\mathbb{R}))$, if $(\alpha,A)$ is subcritical, then it is almost reducible.
\end{theorem}

If we restrict ourself to the quasi-periodic Schr\"odinger cocycle $(\alpha, S_{E}^V)$, which comes from the quasi-periodic
Schr\"odinger operator 
\begin{equation*}
(H_{V,\alpha,\theta}u)_{n}=u_{n+1}+u_{n-1}+V(n\alpha+\theta)u_{n},\ \ \forall \theta \in \T,
\end{equation*}
Then we classify the energy $E\in \Sigma(V)$ by the dynamical behavior of $(\alpha, S_{E}^V)$. We denote $E\in  \Sigma_{-}(V)$ if and only if  $(\alpha, S_{E}^V)$ is subcritical, 
$E\in  \Sigma_{c}(V)$ if and only if  $(\alpha, S_{E}^V)$ is critical,  and $E\in  \Sigma_{+}(V)$ if and only if  $(\alpha, S_{E}^V)$ is supercritical.

%

\section{Explicit formulas of Lyapunov exponent in the spectrum}

It is known that the ac spectrum locates at the place where the Lyapunov exponent is zero, while pp spectrum locates at the place where the Lyapunov exponent is positive. Thus, the key for ME is the exact formula of the Lyapunov exponent. 
For this purpose,  we  consider the cocycle $(\alpha,A(\cdot+i\epsilon))$ with $\epsilon>0$. 
For models considered in this paper, we can reduce the non-trival problem of computing $L(\alpha,A(\cdot))$ to the problem of computing  $\lim_{\epsilon \to \infty} L(\alpha,A(\cdot+i\epsilon))$. The later is much easier.   This approach was based on  Avila's global theory of one-frequency quasi-periodic cocycles \cite{A4}.

\subsection{Lyapunov exponent for the mosaic model.}

Note that for the mosaic model \eqref{model_2}, let (abusing the notation a bit, we still denote it  by $V_2$)
\begin{equation*}
	V_{2}(\theta,n)=\left\{\begin{matrix}2\lambda \cos2\pi\theta,&n\in \kappa \mathbb{Z},\\ 0,&else.\end{matrix}\right.
\end{equation*}
Then $V_2$ is defined on $\mathbb{T}\times\mathbb{Z}_{\kappa}$, consequently \eqref{model_2} induces an almost-periodic Schr\"odinger cocycle $(T_{\alpha},S_{E}^{V_{2}})$
where  $T_{\alpha} (\theta,n)= (\theta+\alpha,n+1)$.  Although $(T_{\alpha},S_{E}^{V_{2}})$ is not a quasi-periodic cocycle in the strict sense, its iterate 
$$(\kappa\alpha, D_{E}^{V_{2}})=: (\kappa\alpha, S_E^{V_2}(\theta,\kappa-1)\cdots S_E^{V_2}(\theta,0))$$   indeed defines an analytic quasi-periodic cocycle.  By simple calculation,
\begin{equation*}
	\begin{split}
		D_{E}^{V_{2}}(\theta)&=\begin{pmatrix}E&-1\\1&0\end{pmatrix}^{\kappa-1}\begin{pmatrix}E-2\lambda \cos2\pi\theta&-1\\1&0\end{pmatrix}\\
		&=\begin{pmatrix}
			a_{\kappa}&-a_{\kappa-1}\\a_{\kappa-1}&-a_{\kappa-2}\end{pmatrix}\begin{pmatrix}E-2\lambda \cos2\pi\theta&-1\\1&0\end{pmatrix}\\&
	\end{split}
\end{equation*}
where
$$a_{\kappa}=
\frac{1}{\sqrt{E^2-4}}\left( (\frac{E+ \sqrt{E^2-4}}{2})^{\kappa}- (\frac{E- \sqrt{E^2-4}}{2})^{\kappa} \right),$$
and $a_{\kappa}(\pm 2)= (-1)^{\kappa-1}\kappa$ by continuity. 
It is easy to see  that $L(T_\alpha,S_{E}^{V_{2}})=\frac{1}{\kappa}L(\kappa\alpha,D_{E}^{V_{2}})$.  The latter can be explicitly computed  by Avila's global theory, thus we have the following result:

\begin{lemma}\label{lemma3.11}
	Suppose that  $\lambda \neq 0$ and  $\alpha\in \R\backslash \Q$. Then for  $E\in\Sigma(V_{2})$,
	$$L(T_\alpha,S_{E}^{V_{2}})=\frac{1}{\kappa}\max\{\ln|\lambda a_{\kappa}|,0\}$$
Moreover, if $E\in\Sigma(V_{2})$, then the cocycle $(\kappa\alpha,D_{E}^{V_{2}})$ is:
\begin{itemize}
	\item supercritical, if and only if $|\lambda a_{\kappa}|>1$,
	\item critical, if and only if $|\lambda a_{\kappa}|=1$,
	\item subcritical, if and only if $|\lambda a_{\kappa}|<1$.
	\end{itemize} 	
\end{lemma}
\begin{proof}
	It suffices  to prove that for any 
	$E\in\Sigma(V_{2}),$ we have
	$$L(\kappa\alpha,D_{E}^{V_{2}})=\max\{\ln|\lambda a_{\kappa}|,0\}.$$
	First we rewrite the matrix $D_E^{V_2}(\theta)$ as
	\begin{equation*}
		D_E^{V_2}(\theta)=\begin{pmatrix}
			a_{\kappa}(E-\lambda(e^{i2\pi\theta}+e^{-i2\pi\theta}))-a_{\kappa-1}&-a_{\kappa}\\a_{\kappa-1}(E-\lambda(e^{i2\pi\theta}+e^{-2\pi\theta}))-a_{\kappa-2}&-a_{\kappa-1}
		\end{pmatrix},
	\end{equation*}
	then we complexify the phase
	\begin{equation*}
	\begin{split}
		(D_{E}^{V_2})_{\epsilon}&=:D_E^{V_2}(\theta+i\epsilon)\\&=\begin{pmatrix}
			a_{\kappa}(E-\lambda(e^{i2\pi(\theta+i\epsilon)}+e^{-i2\pi(\theta+i\epsilon)}))-a_{\kappa-1}&-a_{\kappa}\\a_{\kappa-1}(E-\lambda(e^{i2\pi(\theta+i\epsilon)}+e^{-2\pi(\theta+i\epsilon)}))-a_{\kappa-2}&-a_{\kappa-1}
		\end{pmatrix},
		\end{split}
	\end{equation*}
thus for sufficiently large $\epsilon$
	\begin{equation*}
		D_E^{V_2}(\theta+i\epsilon)=e^{2\pi\epsilon}e^{-i2\pi\theta}\begin{pmatrix}
			-\lambda a_{\kappa}&0\\-\lambda a_{\kappa-1}&0
		\end{pmatrix}+ o(1).
	\end{equation*}

	Let $A=\begin{pmatrix}
		-\lambda a_{\kappa}&0\\-\lambda a_{\kappa-1}&0
	\end{pmatrix}$. 
	Then $
	A^n=(-\lambda)^{n}\begin{pmatrix}
		a_{\kappa}^n&0\\a_{\kappa-1}a_{\kappa}^{n-1}&0
	\end{pmatrix}.
	$
	It is obvious  that
	\begin{equation*}
		\lim_{n\rightarrow\infty} \frac{\ln||A^n||}{n}=\lim_{n\rightarrow\infty}\frac{\ln|(-{\lambda})^{n}a_{\kappa}^n|}{n}
		=\ln|{\lambda a_{\kappa}}|.
	\end{equation*}
	By the continuity of Lyapunov exponent \cite{BJ,jitomirskaya2009continuity}, we have
	\begin{equation*}
		L(\kappa\alpha, (D_{E}^{V_2})_{\epsilon})=2\pi\epsilon+\ln|{\lambda a_{\kappa}}| +o(1).
	\end{equation*}
	By Theorem \ref{theorem3.4}, $\omega(\kappa\alpha, (D_{E}^{V_2})_{\epsilon})=1$ and  
	\begin{equation*}
		L(\kappa\alpha, (D_{E}^{V_2})_{\epsilon})=2\pi\epsilon+\ln|{\lambda a_{\kappa}}|
	\end{equation*} for  sufficiently large $\epsilon$. 
	By real-symmetry,   $\omega(\kappa\alpha, (D_{E}^{V_2})_{\epsilon})$
	is either 0 or 1 for
	$\epsilon\geq 0$.
	This implies that
	\begin{equation}\label{lemosaic} L(\kappa\alpha, (D_{E}^{V_2})_{\epsilon})= \max\{\ln|\lambda a_{\kappa}|+2\pi \epsilon, L(\kappa\alpha, D_{E}^{V_2})\}. 
	\end{equation}
	
	As a consequence, we have  
	\begin{equation*}
		L(\kappa\alpha, D_E^{V_2})\geq  \max\{\ln|\lambda a_{\kappa}|,0\} .
	\end{equation*}
	If $L(\kappa\alpha, D_E^{V_2})> \max\{\ln|\lambda a_{\kappa}|,0\} 
	$, then  $L(\kappa\alpha, D_E^{V_2})>0$ and $\omega(\kappa\alpha, (D_{E}^{V_2})_{\epsilon})=0$ for sufficient small and positive $\epsilon$, which implies that $(\kappa\alpha, D_E^{V_2})$ is   uniformly hyperbolic by Theorem 6 of \cite{A4}, and thus $(T_\alpha,S_{E}^{V_{2}})$   is   uniformly hyperbolic. It contradicts with $E\in \Sigma(V_{2})$. 
	Therefore
	$$L(\kappa\alpha, D_E^{V_2})= \max\{\ln|\lambda a_{\kappa}|,0\} $$ for $E\in\Sigma(V_{2}).$
	Moreover, \eqref{lemosaic} implies that    $(\kappa\alpha, D_E^{V_2})$ is supercritical if and only if $|\lambda a_{\kappa}|>1$. The other cases follow similarly. 
\end{proof}

To locate the spectrum, we need the following observations: 
\begin{lemma}\label{lemma3.13}
	For any $\alpha\in \R\backslash \Q,\kappa\in\mathbb{Z}^+$, we have $\Sigma(V_{2})=-\Sigma(V_{2})$.
	\end{lemma}
	\begin{proof}
		Suppose that $E\in\Sigma(V_{2})$ and $H_{V_{2},\alpha,\theta}u=Eu$, define $\widetilde{u}$ as
			$$\widetilde{u}_{2n}=u_{2n},\ \widetilde{u}_{2n+1}=-u_{2n+1}.$$
			Direct computation shows that  $H_{V_{2},\alpha,\theta+\frac{1}{2}}\widetilde{u}=-E\widetilde{u}$. Then $-E\in\Sigma(V_{2})$ since $\Sigma(V_{2})$ is independence of the phase $\theta$.
	\end{proof}

Lemma \ref{lemma3.13} says that $\Sigma(V_{2})$ is symmetric with respect to $0$, actually we will show that $0$ always belongs to $\Sigma(V_{2})$ (Lemma \ref
{zeroenergy} and Lemma \ref{zeroenergy-1}).  Moreover, by direct calculation, we obtain that $a_{\kappa}$ has $\kappa-1$ roots ($E_l=2\cos\frac{\pi l}{\kappa},l=1,\cdots,\kappa-1$). Thus the set of $E$ satisfying $|\lambda a_{\kappa}|<1$ is a union of at most $\kappa-1$ open intervals $\cup_{j=1}^{m_{\lambda}}(b_j,c_j)$ such that $|a_{\kappa}(b_j)|=|a_{\kappa}(c_j)|=\frac{1}{\lambda},$ $1\le m_{\lambda}\le\kappa-1$, and each of these open intervals $(b_j,c_j)$ has at least one root of $a_{\kappa}$. Obviously, the distance between these $E_l$ is constant, thus $m_{\lambda}=\kappa-1$ if $\lambda$ is sufficiently large.  Next we will show  the roots $E_l=2\cos\frac{\pi l}{\kappa}$ are always in the spectrum, indeed for any $E_l$, the corresponding cocycles $(\kappa\alpha, D_{E_l}^{V_{2}})$ are always reducible:

\begin{lemma}\label{zeroenergy}
	For any $\alpha\in \R\backslash \Q,$ $\lambda\neq 0$, $\kappa\in\mathbb{Z}^+$.  Then there exists $B\in C^{\omega}(\T,SL(2,\R))$ such that 
	$$B(\theta+2\alpha)^{-1}D_{E_l}^{V_{2}}(\theta)B(\theta)=  \begin{pmatrix}\pm 1&0\\0& \pm 1\end{pmatrix}.$$
	In particular,
	one has 
	\begin{equation}\label{non}\Sigma(V_{2})\cap \{ E\in \R \quad |\quad   |\lambda a_{\kappa}|<1\}  \neq\emptyset. \end{equation} \end{lemma}
\begin{proof}
	Note that $a_{\kappa}=a_{\kappa-1}E-a_{\kappa-2}$ and $a_{\kappa}(E_l)=0$, direct computation shows that 
	\begin{equation*}
	\begin{split}
	D_{E_l}^{V_{2}}(\theta)=\begin{pmatrix}-a_{\kappa-1}&0\\-2\lambda a_{\kappa-1}\cos2\pi\theta &-a_{\kappa-1}\end{pmatrix}=-a_{\kappa-1}(E_l)\begin{pmatrix}
	1&0\\2\lambda\cos2\pi\theta&1
	\end{pmatrix},
	\end{split}
	\end{equation*}
	where $a_{\kappa-1}(E_l)=\pm 1$ since $\det D_{E_l}^{V_2}=1$.
	The equation
	\begin{equation*}
		2\lambda \cos2\pi\theta=h(\theta+2\alpha)-h(\theta)
	\end{equation*}
	always has a solution since $\int_{\mathbb{T}}\cos2\pi\theta d\theta=0$ and $\alpha$ is irrational. Let $h(\theta)\in C^{\omega}(\mathbb{T},\mathbb{R})$ be its solution and denote $B(\theta)=\begin{pmatrix}1&0\\h(\theta)&1\end{pmatrix}$, then one can easily check that
	\begin{equation*}
		\begin{split}
			&B(\theta+2\alpha)^{-1}D_{E_l}^{V_{2}}(\theta)B(\theta)\\
			&=-a_{\kappa-1}(E_l)\begin{pmatrix}1&0\\- h(\theta+2\alpha)&1\end{pmatrix}\begin{pmatrix}1&0\\2\lambda \cos2\pi\theta&1\end{pmatrix}\begin{pmatrix}1&0\\h(\theta)&1\end{pmatrix}\\
			&=\begin{pmatrix} \pm1&0\\0& \pm1\end{pmatrix},
		\end{split}
	\end{equation*}
 This implies that $L(T_\alpha, S_{E_l}^{V_2})=\frac{1}{\kappa}L(\kappa\alpha,D_{E_l}^{V_{2}})=0$, and  $E_l \in \Sigma(V_2)$, then  \eqref{non} follows directly. \end{proof}

As a direct consequence, if $\kappa$ is  an even number, then $0\in \Sigma(V_2)$, if $\kappa$ is an odd number, then we have the following observation:

\begin{lemma}\label{zeroenergy-1}
	For any $\alpha\in \R\backslash \Q,$ $\lambda\neq 0$. If $\kappa\in 2\mathbb{Z}+1$, then  $0\in  \Sigma(V_2)$.	
	\end{lemma}
\begin{proof}
If $\kappa\in 2\mathbb{Z}+1$, and $E=0$,  by simple calculation,
$$(\kappa \alpha, D_{0}^{V_{2}}(\theta))=(\kappa \alpha,  \pm \begin{pmatrix}-2\lambda \cos2\pi\theta&-1\\1&0\end{pmatrix}), $$
which is just the almost Mathieu cocycle.  Thus $0\in  \Sigma(V_2)$ is equivalent to whether $0$ belongs to the spectrum of almost Mathieu operator:
\begin{equation*}
(H_{\lambda,\kappa\alpha,\theta} u)_n= u_{n+1}+u_{n-1} +2\lambda \cos 2
\pi (n\kappa\alpha + \theta) u_n,
\end{equation*} then the result follows from \cite{BKR}	directly.
\end{proof}

Now we summarize the above results in the case $\kappa=2$ and $\kappa=3$,  which can be clarified very clearly:

\begin{corollary}\label{moscor}
	Suppose that $\lambda \neq 0$, $\alpha\in \R\backslash \Q$, and $\kappa=2$. Then we have
	$$L(T_\alpha,S_{E}^{V_{2}})=\frac{1}{2}\max\{\ln|\lambda E|,0\},\ E\in\Sigma(V_{2}).$$
	Moreover, the following holds true:
	\begin{enumerate}
		\item If $\lambda\overline{E}(V_2)<1$, then for any $E\in \Sigma(V_{2})$,  $(2\alpha,D_{E}^{V_{2}})$ is subcritical.
		\item If $\lambda\overline{E}(V_2)>1$, then we have the following:
		\begin{itemize}
			\item   $ \Sigma(V_{2})\cap(-\frac{1}{\lambda},\frac{1}{\lambda})  \neq\emptyset$,  furthermore $(2\alpha,D_{E}^{V_{2}})$ is subcritical.
			\item   $ \Sigma(V_{2})\cap[-\frac{1}{\lambda},\frac{1}{\lambda}]^c  \neq\emptyset$,   furthermore $(2\alpha,D_{E}^{V_{2}})$ is supercritical.
		\end{itemize}
	\end{enumerate}
\end{corollary}

\begin{proof}
	Just note $a_{2}=E$, then it is  direct consequences of  Lemma \ref{lemma3.11} and Lemma  \ref{zeroenergy}. 
\end{proof}

Lemma \ref{moscor} ensures  $\Sigma(V_{2})\cap(-\frac{1}{\lambda},\frac{1}{\lambda})\neq\emptyset$
for any $\lambda.$  Note if $\lambda$ is small enough, then 
$$\Sigma(V_{2})\subseteq[-2-2\lambda,2+2\lambda] \subset (-\frac{1}{\lambda},\frac{1}{\lambda}),$$
which means Corollary \ref{moscor} $(1)$ holds for small $\lambda$ and there is no ME.  It follows that  ME appears only when $\lambda$ is relatively large. However,    $\frac{1}{\lambda}<\overline{E}(V_2)$ is not easy to be verified since  $\overline{E}(V_2)$ depends implicitly on  $\lambda$. Next result will tell us, Corollary \ref{moscor}  $(2)$ holds at least for  $\lambda>\frac{\sqrt{2}}{2}:$

\begin{lemma}
If $\lambda>\frac{\sqrt{2}}{2},$ $\kappa=2$, then $\Sigma(V_{2})\cap([-2-2\lambda,-\frac{1}{\lambda})\cup(\frac{1}{\lambda},2+2\lambda])\neq\emptyset$ 
\end{lemma} 
\begin{proof}
	We  prove  $\Sigma(V_{2})\cap(\frac{1}{\lambda},2+2\lambda]\neq\emptyset$. 
	First by the spectral theorem,  we have 
	\begin{equation*}
		2\lambda\cos2\pi(\theta+2n\alpha)= \left<\delta_{2n},H_{V_{2},\alpha,\theta}\delta_{2n}\right> =\int Ed\mu_{V_{2},\alpha,\theta,\delta_{2n}}.
	\end{equation*}
	One can  select $n$ such that $\cos2\pi(2n\alpha+\theta)>\frac{1}{2\lambda^2}$, such $n$ exists since $\alpha$ is irrational  and $\frac{1}{2\lambda^2}<1$. On 
	the other hand, suppose that $\Sigma(V_{2})\cap(\frac{1}{\lambda},2+2\lambda]=\emptyset$, i.e. for any $E\in \Sigma(V_{2})$, we have $|E|\leq \frac{1}{\lambda}$. 
This immediately imply that
	$$\frac{1}{\lambda}<2\lambda \cos2\pi(2n\alpha+\theta)=\int Ed\mu_{V_{2},\alpha,\theta,\delta_{2n}}\le\frac{1}{\lambda}.$$
	This is a contradiction. 
\end{proof}

\begin{remark}
This argument can be generalized to general $\kappa$ almost without change. Consequently, combining Lemma \ref{zeroenergy}, if $|\lambda|$ is relatively large, then ME always exists.   
\end{remark}

In the case $\kappa=3$, recall that 
$$E_{c}^1=   \sqrt{1+ \frac{1}{\lambda}}, \qquad   E_{c}^2=   \sqrt{1- \frac{1}{\lambda}},$$
 then we have the  following:

\begin{corollary}\label{moscor-k3}
	Suppose that $\lambda \neq 0$, $\alpha\in \R\backslash \Q$, and $\kappa=3$. Then we have
	$$L(T_\alpha,S_{E}^{V_{2}})=\frac{1}{3}\max\{\ln|\lambda (E^2-1)|,0\},\ E\in\Sigma(V_{2}).$$
	Moreover, the following holds true:
	\begin{enumerate}
		\item If $|\lambda| (\overline{E}(V_2)^2-1)< 1$, then for any $E\in \Sigma(V_{2})$,  $(3\alpha,D_{E}^{V_{2}})$ is subcritical.
		\item If  $ \frac{1}{ \overline{E}(V_2)^2-1 } < |\lambda| < 1$, then we have the following:
		\begin{itemize}
			\item  $\Sigma(V_2)\cap (- E_{c}^1, E_{c}^1) \neq\emptyset$,  furthermore $(3\alpha,D_{E}^{V_{2}})$ is subcritical.
			\item   $\Sigma(V_2)\cap [- E_{c}^1, E_{c}^1]^c  \neq\emptyset$,   furthermore $(3\alpha,D_{E}^{V_{2}})$ is supercritical.
		\end{itemize}
		\item If  $ |\lambda| >1$, then we have the following:
		\begin{itemize}
			\item  $\Sigma(V_2)\cap ( E_c^2, E_c^1)  \neq\emptyset$ and  $\Sigma(V_2)\cap (-E_c^1, E_c^2) \neq\emptyset$,  furthermore $(3\alpha,D_{E}^{V_{2}})$ is subcritical.
			\item  $\Sigma(V_2)\cap [- E_c^1,   E_c^1]^c\neq\emptyset $  and  $\Sigma(V_2)\cap ( -E_c^2, E_c^2) \neq\emptyset$,   furthermore $(3\alpha,D_{E}^{V_{2}})$ is supercritical.
		\end{itemize}
	\end{enumerate}
\end{corollary}

\begin{proof}
	Just note $a_{3}=E^2-1$, then it is  direct consequences of  Lemma \ref{lemma3.11},  Lemma  \ref{zeroenergy} and Lemma  \ref{zeroenergy-1}. 
\end{proof}

\subsection{Lyapunov exponent for the GAA model}

Similar to  the proof of Lemma \ref{lemma3.11}, we can also calculate the Lyapunov exponent for the GAA model:

\begin{lemma}\label{lemma3.7}
	Suppose that $\lambda \in \R$,  $\alpha\in \R\backslash \Q$ and $|\tau|\le1$. Then, for any $E\in\Sigma(V_{1})$, we have 
	\begin{equation}\label{legaa}L(\alpha,S_{E}^{V_{1}})= \max\{\ln|\frac{|\tau E+2\lambda|+\sqrt{(\tau E+2\lambda)^2-4\tau^2}}{2(1+\sqrt{1-\tau^{2}})}|,0\}.\end{equation}
	Moreover, we have the following:
	\begin{itemize}
		\item $E\in \Sigma_{+}(V_1)$,  if and only if $sgn(\lambda)\tau E>2(1-|\lambda|)$,
		\item $E\in \Sigma_{c}(V_1)$,  if and only if $sgn(\lambda)\tau E=2(1-|\lambda|)$,
		\item $E\in \Sigma_{-}(V_1)$,  if and only if $sgn(\lambda)\tau E<2(1-|\lambda|)$.
	\end{itemize}   
\end{lemma}

\begin{proof}
	We distinugish three cases:
	$\tau=0$, $0<|\tau|<1$ and $|\tau|=1$.
	
	\textbf{Case 1:} If $\tau=0$, the result follows from \cite{A4}, or  Lemma \ref{lemma3.11} with $\kappa=1$.
	
	\textbf{Case 2:} If $0<|\tau|<1$, the potential is bounded and analytic. Let $D(\theta) : =  2(1-\tau \cos2\pi(\theta)) S_{E}^{V_{1}}(\theta)$, i.e.,
	\begin{eqnarray*}
		D(\theta) : =\begin{pmatrix}2E(1-\tau \cos2\pi\theta)-4\lambda\cos2\pi\theta& -2+2\tau \cos2\pi\theta\\ &\\2-2\tau \cos2\pi\theta&0\end{pmatrix},
	\end{eqnarray*}
then $D(\theta)$ admits a holomorphic extension to $|\Im\theta|<\infty$. Note that 
	\begin{equation*}
		\begin{split}
			f(\theta,\epsilon)&=2(1-\tau \cos2\pi(\theta+\epsilon i))\\
			&=2-\tau e^{-2\pi\epsilon}e^{2\pi\theta i}-\tau e^{2\pi\epsilon}e^{-2\pi\theta i}\\
			&=-\tau e^{-2\pi\epsilon-2\pi\theta i}(e^{2\pi\theta i}-\frac{1+\sqrt{1-\tau^2}}{\tau}e^{2\pi\epsilon})(e^{2\pi\theta i}-\frac{1-\sqrt{1-\tau^2}}{\tau}e^{2\pi\epsilon}).
		\end{split}
	\end{equation*}
	Thus $S_{E}^{V_{1}}\in C^{\omega}(\mathbb{T},SL(2,\mathbb{R}))$ admits a holomorphic extension to the strip $|\Im\theta|<\delta_{0}$ with 
	$\delta_{0}=\frac{1}{2\pi}\ln\frac{1+\sqrt{1-\tau^{2}}}{|\tau|}$.

	By Jensen's formula, we have 
	\begin{equation*}
		\int_{\mathbb{T}}\ln|f(\theta,\epsilon)|d\theta=\ln(1+\sqrt{1-\tau^{2}}), \quad \forall \, |\epsilon|<\delta_{0}.
	\end{equation*}
	Therefore we have
	\begin{equation}\label{equation_11}
		\begin{split}
			L(\alpha,D_{\epsilon}) &= L(\alpha,(S_{E}^{V_{1}})_{\epsilon})+ \int_{\mathbb{T}}\ln|f(\theta,\epsilon)|d\theta\\&=L(\alpha,(S_{E}^{V_{1}})_{\epsilon})+\ln(1+\sqrt{1-\tau^{2}}),\quad  \forall \, |\epsilon|<\delta_{0},
		\end{split}
	\end{equation}
	which implies that $(\alpha,(S_{E}^{V_{1}})_{\epsilon})$ and  
	$(\alpha,D_{\epsilon})$ have the same acceleration when $|\epsilon|<\delta_{0}$, i.e. 
	\begin{equation}\label{equation_12}
		\omega(\alpha,D_{\epsilon})= \omega(\alpha,(S_{E}^{V_{1}})_{\epsilon}),\quad \forall |\epsilon|<\delta_{0}.
	\end{equation}
	
	On the other hand, if we  complexify the phase, and write 
	\begin{equation*}
		\begin{split}
			D(\theta+\epsilon i)=(e^{-2\pi\theta i+2\pi\epsilon}+e^{2\pi\theta i-2\pi\epsilon})\begin{pmatrix}-(\tau E+2\lambda)&\tau\\ -\tau&0\end{pmatrix} +
			\begin{pmatrix}2E&-2\\ 2&0 \end{pmatrix}.
		\end{split}
	\end{equation*}
	Let $\epsilon$ goes to infinity, then
	\begin{equation*}
		D(\theta+\epsilon i)=e^{-2\pi\theta i+2\pi\epsilon}\begin{pmatrix}-(\tau E+2\lambda)+o(1)&\tau+o(1)\\ -\tau+o(1)&0\end{pmatrix}.\end{equation*}
	By the continuity of Lyapunov exponent \cite{BJ,jitomirskaya2009continuity}, we have
	\begin{equation*}
		L(\alpha,D_{\epsilon})=\ln|h(E)|+2\pi\epsilon+o(1),
	\end{equation*}
	where
	$$h(E)=\frac{1}{2}(|\tau E+2\lambda|+\sqrt{(\tau E+2\lambda)^2-4\tau^2}).$$
	By quantization of acceleration(Theorem \ref{theorem3.4}), 
	$$L(\alpha,D_{\epsilon})=\ln|h(E)|+2\pi\epsilon \quad\text{for all $\epsilon$ sufficiently large}, $$
	which also implies that $\omega(\alpha,D_{\epsilon})=1$ for  sufficiently large $\epsilon$.  
	
	By convexity,   $\omega(\alpha,D_{\epsilon})\leq 1$ for every $\epsilon\geq 0$. just note $D \notin SL(2,\mathbb{C})$, in general one cann't conclude  $\omega(\alpha,D_{\epsilon})=0$ or $1$ for every $\epsilon\geq 0$. Nevertheless, since $S_{E}^{V_{1}}\in C^{\omega}(\mathbb{T},SL(2,\mathbb{R}))$,  again by   Theorem \ref{theorem3.4}, one has $\omega(\alpha,(S_{E}^{V_{1}})_{\epsilon}) \in \Z$  for any $|\epsilon|<\delta_{0}.$ Thus if   $E\in\Sigma_+(V_{1})$ or $E\in\Sigma_c(V_{1})$, then  by  \eqref{equation_12}, and the convexity of $L(\alpha,D_{\epsilon})$, we have $\omega(\alpha,(S_{E}^{V_{1}})_{\epsilon})=1$ for $0\le\epsilon<\delta_{0}$ and $\omega(\alpha,D_{\epsilon})=1$ for $\epsilon\ge0$. As a consequence, it holds that
	\begin{equation*}
		L(\alpha,D_{\epsilon})=\ln|h(E)|+2\pi|\epsilon|, \quad \text{for all}\ \epsilon.
	\end{equation*}
	where the case $\epsilon\leq 0$ follows by real-symmetry.   By  \eqref{equation_11} and the non-negativity of $L(\alpha,S_{E}^{V_{1}})$, if    $E\in\Sigma_+(V_{1})$ or $E\in\Sigma_c(V_{1})$, we have
	\begin{equation}\label{equation_16}
L(\alpha,(S_{E}^{V_{1}})_\epsilon)=\ln \frac{|h(E)|}{1+\sqrt{1-\tau^2}}+2\pi|\epsilon|,\quad|\epsilon|<\delta_{0}.
	\end{equation}
	If $E\in\Sigma_-(V_{1})$,  then $L(\alpha,(S_{E}^{V_{1}})_\epsilon)=0$ for $|\epsilon|<\delta' \leq \delta_{0}$. 
	
	By Avila's global theory \cite{A4}, for any $E\in\Sigma(V_{1})$, the corresponding  cocycle $(\alpha,S_{E}^{V_{1}})$ is either supercritical, critical, or subcritical, we thus only need to locate the energy which is supercritical or critical.   Without losing generality,  we assume  $\lambda<0$, $\tau>0$. By $\eqref{equation_16}$, 
	$E\in\Sigma_c(V_{1})$, if and only if $|h(E)|=1+\sqrt{1-\tau^2}$, which is equivalent to $sgn(\lambda)\tau E=2(1-|\lambda|)$ by simple  calculation. Meanwhile, 
	$E\in\Sigma_+(V_{1})$, if and only if $|h(E)|>1+\sqrt{1-\tau^2}$, which is equivalent to  $|\tau E+2\lambda|>2$. In our case  $\lambda<0$ and $\tau>0$,  we actually have $\tau E+2\lambda< -2$ since $\tau E+2\lambda>2$ is impossible. In fact if $\tau E+2\lambda>2$  then 
	$$E>\frac{2-2\lambda}{\tau}> 2+\frac{2\lambda}{1-\tau},$$
	which contradicts with the fact 
	$$\Sigma(V_{1})\subseteq[-2-\frac{2\lambda}{1+\tau},2+\frac{2\lambda}{1-\tau}].$$

	\textbf{Case 3:} In the limiting case $|\tau|=1$, the operator is unbounded. However, recall that $(\alpha,D)$ is a bounded and analytic cocycle, thus   $L(\alpha,D_{\epsilon})$ is  continuous in $\epsilon$. Moreover, 
	$$ L(\alpha,D_{\epsilon}) = L(\alpha,(S_{E}^{V_{1}})_{\epsilon})+ \int_{\mathbb{T}}\ln|f(\theta,\epsilon)|d\theta,$$
	applying Jensen's formula yields $$\int_{\mathbb{T}}\ln|f(\theta,\epsilon)|d\theta=2\pi|\epsilon|,\quad \text{for all}\ \epsilon.$$
	Since the above equation explicitly implies the continuity of $\int_{\mathbb{T}}\ln|f(\theta,\epsilon)|d\theta$ in $\epsilon$, the continuity of $L(\alpha,(S_E^{V_1})_{\epsilon})$ follows.
	
	Then, uniformly in $\theta\in\mathbb{T}$, one has
	$$(S_{E}^{V_1})_{\epsilon}=B_{\infty}+o(1)$$
	as  $\epsilon$ goes to infinity, where
	$$B_{\infty}=\begin{pmatrix}
		E+2\lambda \tau&-1\\1&0
	\end{pmatrix}.$$
	By continuity of the Lyapunov exponent \cite{BJ,jitomirskaya2009continuity}, we have
	$$L(\alpha,(S_E^{V_1})_{\epsilon})=L(\alpha,B_{\infty})+o(1)$$
	as  $\epsilon$ goes to infinity.
	
	The quantization of acceleration (Theorem \ref{theorem3.4}) yields
	$$L(\alpha,(S_E^{V_1})_{\epsilon})=L(\alpha,B_{\infty})\quad\text{for all $\epsilon$ sufficiently large}.$$
	In addition, the convexity, continuity, and symmetry of $L(\alpha,(S_E^{V_1})_{\epsilon})$ with respect to $\epsilon$ gives
	$$L(\alpha,(S_E^{V_1})_{\epsilon})=L(\alpha,B_{\infty})\quad\text{for all $\epsilon$}. $$
	This actually implies that
	$$L(\alpha,S_E^{V_1})=\max\{\frac{|E+2\lambda \tau|+\sqrt{(E+2\lambda \tau)^2-4}}{2},0\},$$
	which is exactly \eqref{legaa}, since $|\tau|=1$.
		\end{proof}

As a direct consequence of Lemma \ref{lemma3.7}, we have

\begin{corollary}\label{theorem3.9}
	Suppose that $\lambda \in \R$,  $\alpha\in \R\backslash \Q$ and $|\tau|<1$. Then if $\lambda\tau>0$,  the following holds true: 
	\begin{enumerate}
		\item If $|\lambda|<1-\frac{|\tau|}{2}\overline{E}(V_1)$, then we have  $$\Sigma (V_{1})=  \Sigma_{-}(V_1)   \neq  \emptyset.$$
		\item If $1-\frac{|\tau|}{2}\overline{E}(V_1)<|\lambda|<1-\frac{|\tau|}{2}\underline{E}(V_1)$, then  we have 
		\begin{equation}\label{spgaa}
			\begin{split}
				\Sigma(V_1)\cap[\underline{E}(V_1),\frac{2(1-|\lambda|)}{|\tau|})&=\Sigma_-(V_1) \neq \emptyset, \\
				\Sigma(V_1)\cap(\frac{2(1-|\lambda|)}{|\tau|},\overline{E}(V_1)]&=\Sigma_+(V_1)\neq \emptyset.
			\end{split}
		\end{equation}
		\item If $|\lambda|>1-\frac{|\tau|}{2}\underline{E}(V_1)$, then  we have $$\Sigma (V_{1})=  \Sigma_{+}(V_1)   \neq  \emptyset.$$
	\end{enumerate}
\end{corollary} 	

\begin{proof}
	By Lemma \ref{lemma3.7}, if $\lambda\tau>0$, then $E\in\Sigma_+(V_1)$ if and only if $|\tau|E>2(1-|\lambda|)$, and $\frac{2(1-|\lambda|)}{|\tau|}$ is the critical point. One can distinguish the following three cases:
	\begin{itemize}
		\item  if $\underline{E}(V_1)>\frac{2(1-|\lambda|)}{|\tau|}$, then $\Sigma(V_1)=\Sigma_+(V_1)$;
		\item if $\overline{E}(V_1)<\frac{2(1-|\lambda|)}{|\tau|}$, then $\Sigma(V_1)=\Sigma_-(V_1)$; 
		\item if $\underline{E}(V_1)<\frac{2(1-|\lambda|)}{|\tau|}<\overline{E}(V_1)$, then \eqref{spgaa} holds.
	\end{itemize}
	We thus finish the whole proof. 
\end{proof}

We have similar results for $\lambda\tau<0$.

\begin{corollary}\label{theorem3.9-1}
	Suppose that $\lambda \in \R$,  $\alpha\in \R\backslash \Q$ and $|\tau|<1$. Then  if $\lambda\tau<0$,  the following holds true: 
	\begin{enumerate}
		\item If $|\lambda|<1+\frac{|\tau|}{2}\underline{E}(V_1)$, then we have  $$\Sigma (V_{1})=  \Sigma_{-}(V_1)   \neq  \emptyset.$$
		\item If $1+\frac{|\tau|}{2}\underline{E}(V_1)<|\lambda|<1+\frac{|\tau|}{2}\overline{E}(V_1)$, then  we have $$\Sigma(V_1)\cap(-\frac{2(1-|\lambda|)}{|\tau|},\overline{E}(V_1)]=\Sigma_-(V_1) \neq \emptyset, $$  $$\Sigma_+(V_1)\cap[\underline{E}(V_1),-\frac{2(1-|\lambda|)}{|\tau|})=\Sigma_+(V_1)\neq \emptyset.$$
		\item If $|\lambda|>1+\frac{|\tau|}{2}\overline{E}(V_1)$, then  we have $$\Sigma (V_{1})=  \Sigma_{+}(V_1)   \neq  \emptyset.$$
	\end{enumerate}
\end{corollary} 

\begin{proof}
	We  omit the details, since the proof is the same as Corollary \ref{theorem3.9}.
\end{proof}
%
%

In general, since $\underline{E}(V_{1})$ and $\overline{E}(V_{1})$ depend on the potential  $V_1$ implicitly, one does not exactly know  when  $1-\frac{|\tau|}{2}\overline{E}(V_1)<|\lambda|<1-\frac{|\tau|}{2}\underline{E}(V_1)$ (or $1+\frac{|\tau|}{2}\underline{E}(V_1)<|\lambda|<1+\frac{|\tau|}{2}\overline{E}(V_1)$)  does happen, thus one does not exactly know when subcritical and supercritical energies coexist. However, we have the following:

\begin{lemma}\label{lemma3.9}
	Suppose that $\lambda \in \R$,  $\alpha\in \R\backslash \Q$ and $|\tau|<1$. Then the following holds true: 
	\begin{itemize}
		\item If $|\lambda|<(1-|\tau|)^2$, then we have  $$\Sigma_{-}(V_1)\cap \Sigma (V_{1})=  \Sigma_{-}(V_1)   \neq  \emptyset.$$
		\item If $1-|\tau|<|\lambda|<1+|\tau|$, then  we have
	\begin{eqnarray}\label{gaa-}\Sigma(V_1)\cap[\underline{E}(V_1),\frac{2(1-|\lambda|)}{|\tau|})&=&\Sigma_-(V_1) \neq \emptyset, \\ \label{gaa+}\Sigma(V_1)\cap(\frac{2(1-|\lambda|)}{|\tau|},\overline{E}(V_1)]&=&\Sigma_+(V_1)\neq \emptyset.\end{eqnarray}
		\item If $|\lambda|>(1+|\tau|)^2$, then  we have $$\Sigma_{+}(V_1)\cap \Sigma (V_{1})=  \Sigma_{+}(V_1)   \neq  \emptyset.$$
	\end{itemize}
\end{lemma}   
\begin{proof}
	For simplicity, we only consider  the case  $\lambda>0$ and $\tau>0$, the other cases can be dealt with similarly.  First note we have a trivial  bound  $$-2-\frac{2\lambda}{1+\tau} \leq \underline{E}(V_{1}) \leq \overline{E}(V_{1})\le2+\frac{2\lambda}{1-\tau},$$  then the first statement and the third statement follows immediately from Corollary \ref{theorem3.9}.
	
	We are left to prove the second statement, and we only prove \eqref{gaa-}, since \eqref{gaa+} can be proved similarly. 
	First by the spectral theorem, 
	\begin{equation*}
		\begin{split}
			\frac{2\lambda \cos(2\pi(n\alpha+\theta))}{1-\tau \cos(2\pi(n\alpha+\theta))}=\left <\delta_{n},H_{V_{1},\alpha,\theta}\delta_{n}\right >=\int Ed\mu_{V_{1},\alpha,\theta,\delta_{n}}.
		\end{split}
	\end{equation*}
	We argue by contradiction, assume that $\Sigma_{-}(V_1)\cap \Sigma (V_{1}) = \emptyset$, then by Lemma \ref{lemma3.7}, we have $E\ge\frac{2(1-\lambda)}{\tau}$ for every $E\in\Sigma(V_{1})$. Select $n$ such that $\cos(2\pi(n\alpha+\theta))<\frac{1-\lambda}{\tau}$, such $n$ exists since $\alpha$ is irrational and   $\frac{1-\lambda}{\tau}>-1$, as a consequence, 
	\begin{equation*}
		\begin{split}
			\frac{2(1-\lambda)}{\tau}>\frac{2\lambda \cos(2\pi(n\alpha+\theta))}{1-\tau \cos(2\pi(n\alpha+\theta))}=\int Ed\mu_{V_{1},\alpha,\theta,\delta_{n}}\ge\frac{2(1-\lambda)}{\tau}.
		\end{split}
	\end{equation*}
	This is a contradiction.  
\end{proof}	
	

\subsection{Application for the long-range tight-binding model.}

Now we consider the long-range tight-binding model \eqref{model_3}. By Aubry duality, the dual model of \eqref{model_3} can be written as  
\begin{equation*}
	(H_{V_{3},\alpha,\theta}u)_{n}=u_{n+1}+u_{n-1}+ V_3(\theta+n\alpha) u_{n},
\end{equation*}
where 
\begin{equation*}
	\begin{split}
		V_{3}(\theta)&=\frac{2}{\lambda}\sum_{j\neq 0}e^{-p|j|}\cos(2\pi j\theta)=\frac{2}{\lambda}\Im(\sum_{j\ne0}e^{-p|j|}e^{2\pi j\theta i}) \\
		&=\frac{4}{\lambda}\frac{-e^{-2p}+e^{-p}\cos2\pi\theta}{1+e^{-2p}-2e^{-p}\cos2\pi\theta}.
	\end{split}
\end{equation*}
Furthermore, by  Aubry duality \cite{GJLS},  we have 
$$ \Sigma(\widehat H_{V_3,\alpha,\theta} ) =\frac{\lambda}{2}\Sigma(V_3).$$
Then as a direct consequence of Lemma \ref{lemma3.7}, one obtains:

\begin{corollary}\label{corollary3.11}
	Suppost that  $\lambda \neq 0$,  $\alpha\in \R\backslash \Q$ and  $p>0$. Then we have
	\begin{equation*}
		L(\alpha,S_{E}^{V_{3}})=\max\{\ln|\frac{e^{-p}}{2}(|E+\frac{2}{\lambda}|+\sqrt{(E+\frac{2}{\lambda})^2-4})|,0\},\ E\in\Sigma(V_{3}).
	\end{equation*}
	Moreover,  we have the following:
	\begin{itemize}
	\item $E\in \Sigma_{+}(V_3)$ if and only if $sgn(\lambda)E>2\cosh p-\frac{2}{|\lambda|}$,
		\item $E\in \Sigma_{c}(V_3)$ if and only if $sgn(\lambda)E=2\cosh p-\frac{2}{|\lambda|}$,
		\item $E\in \Sigma_{-}(V_3)$ if and only if $sgn(\lambda)E<2\cosh p-\frac{2}{|\lambda|}$.
		\end{itemize}
\end{corollary}

\begin{proof}
	Note that
	\begin{equation*}
		\frac{4}{\lambda}\frac{-e^{-2p}+e^{-p}\cos(2\pi n\alpha+\theta)}{1+e^{-2p}-2e^{-p}\cos(2\pi n\alpha+\theta)}=E_{0}+2\widetilde\lambda\frac{\cos(2\pi n\alpha+\theta)}{1-\tau \cos(2\pi n\alpha+\theta)},
	\end{equation*}
	thus we can rewrite $H_{V_3,\alpha,\theta}u=Eu$ as 
	\begin{equation*}
		u_{n+1}+u_{n-1}+2\widetilde{\lambda}\frac{\cos(2\pi n\alpha+\theta)}{1-\tau \cos(2\pi n\alpha+\theta)}u_{n}=(E-E_{0})u_{n},
	\end{equation*}
	where $$E_{0}=-\frac{2e^{-p}}{\lambda\cosh p}, \qquad \widetilde{\lambda}=\frac{\tanh p}{\lambda\cosh p}, \qquad \tau=\frac{1}{\cosh p}.$$
	Then the results follow from  Lemma \ref{lemma3.7} directly. 
\end{proof}

\subsection{Application to the  ``Peaky" potential}

\begin{corollary}\label{corollary3.13}
	Suppose that  $K\neq0$ and $\alpha\in \R\backslash \Q$. Then  we have
	$$ L(\alpha,S_{E}^{V_4})=\max\{\ln|\frac{K}{2K+1}\frac{|E|+\sqrt{E^2-4}}{1+\sqrt{1-(\frac{2K}{2K+1})^2}}|,0\},\quad E\in\Sigma(V_4).$$
	Moreover,  the following holds true: 
	\begin{enumerate}
		\item If $  \frac{\lambda K}{(2K+1)^2} < 1- \frac{K  \overline{E}(V_{4}) }{2K+1}$, then we have  $$\Sigma (V_{4})=  \Sigma_{-}(V_4)   \neq  \emptyset.$$
		\item If $  1- \frac{\lambda  \overline{E}(V_{4}) }{2K+1}  <  \frac{\lambda K}{(2K+1)^2} <  1- \frac{K \underline{E}(V_{4}) }{2K+1}  $, then  we have 
		\begin{equation}\label{sppeaky}
			\begin{split}
				\Sigma(V_4)\cap[\underline{E}(V_4),  2+\frac{1}{K})&=\Sigma_-(V_4) \neq \emptyset, \\
				\Sigma(V_4)\cap(2+\frac{1}{K}, \overline{E}(V_4)]&=\Sigma_+(V_4)\neq \emptyset.
			\end{split}
		\end{equation}
		\item If $ \frac{\lambda K}{(2K+1)^2} > 1- \frac{\lambda \underline{E}(V_{4}) }{2K+1} $, then  we have $$\Sigma (V_{4})=  \Sigma_{+}(V_4)   \neq  \emptyset.$$
	\end{enumerate}
	In particular, if $1<  \frac{\lambda K}{(2K+1)}<4K+1$, then \eqref{sppeaky} holds. 
\end{corollary}	

%
%

\begin{proof}
	Note one can rewrite the ``Peaky" potential  as 
	$$  V_4(\theta)=\frac{\lambda}{1+4K \sin^2(\pi \theta) }= \frac{\lambda}{2K+1} + \frac{2\lambda K}{(2K+1)^2} \frac{\cos 2\pi \theta}{1- \frac{2K}{2K+1} \cos 2\pi\theta }$$
	Thus the corresponding  Schr\"odinger  operator is in fact  the GAA model:
	\begin{equation*}
		u_{n+1}+u_{n-1}+2\widetilde{\lambda}\frac{\cos2\pi (n\alpha+\theta)}{1-\tau \cos 2\pi (n\alpha+\theta)}u_{n}=(E- \frac{\lambda}{2K+1} )u_{n},
	\end{equation*}
	with
	$ \widetilde{\lambda} = \frac{\lambda K}{(2K+1)^2},   \tau=\frac{2K}{2K+1}, $
	and a shift of energy  $\frac{\lambda}{2K+1}$.
	Then the results follows from Lemma \ref{lemma3.7}, Corollary \ref{theorem3.9} and  Lemma \ref{lemma3.9}. 
\end{proof}

\section{Pure absolutely continuous spectrum}

For any $\alpha\in \R\backslash \Q$,  $\kappa\in\Z^+$, and for any $V^j\in C^{\omega}_h(\T,\R)$, $j=0,1,\cdots \kappa-1,$  we consider the almost-periodic Schr\"odinger operator 
\begin{equation}\label{msch}(H_{V,\alpha, \theta} u)_n= u_{n+1}+u_{n-1}+V_{\theta}(n)u_{n},\end{equation} where the potential takes the form  
$$V_{\theta}(n) = V^{j}(\theta+ n\alpha) \qquad n \equiv  j \mod \kappa.$$
The case  $\kappa=1$ is the  one-frequency  quasi-periodic Schr\"odinger  operator, while the case  $\kappa\geq 2$, including the quasi-periodic mosaic model,  is almost periodic with frequency modulo  $\mathbb{T}\times\mathbb{Z}_{\kappa}$, consequently \eqref{msch} induces an almost-periodic Schr\"odinger cocycle $(T_{\alpha},S_{E}^{V})$, and associate with it, one can consider the quasi-periodic cocycle
\begin{eqnarray*}(\kappa\alpha, D_{E}^{V}) &:=& (\kappa\alpha, S_E^{V}(\theta,\kappa-1)\cdots S_E^{V}(\theta,0))\\
&=& \begin{pmatrix}E- V^{\kappa-1}(\theta+ (\kappa-1)\alpha)&-1\\1&0\end{pmatrix}    \cdots   \begin{pmatrix}E- V^0(\theta)&-1\\1&0\end{pmatrix} .  \end{eqnarray*} 
In general, for this kind of potentials,  we have the following:

\begin{theorem}\label{pac}
If $\alpha\in DC$, $\kappa \in \Z^+$, then for any $\theta \in \R$, $H_{V,\alpha, \theta}$ is purely absolutely continuous in the set 
$$\mathcal{S}=\{E \in \Sigma(V) \quad | \quad (\kappa \alpha, D_E^V) \quad \text{is almost reducible}\}.$$
\end{theorem}

This result  establish the link between absolutely continuous spectrum  of almost periodic operators and almost reducibility  of its iterated cocycle.   If $\kappa=1$, then this is well-known result of Avila \cite{avilalyapunov}.  However, the real challenge is the case $\kappa\geq 2$, and so far none of existence approaches can be applied to this situation. As a matter of fact,
there are basically two existing approach in proving pure ac spectrum based on almost reducibility, which are developed by Eliasson \cite{E92} and Avila \cite{A01} separately. Eliasson's proof  \cite{E92} based on  parameterized KAM, and to study almost reducibility of $(\alpha, A(E-\bar E) e^{F_{E}(\theta)})$,  his approach strongly depends on the fact the constant part $A(E)$ is non-degenerated, i.e., 
\begin{equation}\label{non-deg}
\frac{d}{dE}\rho(\alpha,A(E))|_{E=\bar E} \neq 0,
\end{equation}
as also explored by Bjerkl\"ov and Krikorian (Theorem 2.2 of \cite{BK}), however neither in the case $\lambda$ is large (one may first apply Lemma \ref{zeroenergy} to reduce it to local situation) nor in the case $\lambda$ is small enough, \eqref{non-deg} is satisfied.  Avila's approach \cite{A01} is duality based where the desired almost reducibility estimates are provided by almost localization of the dual operator. However, it is obvious that   we do not have the duality approach  for $\kappa\ge 2$ case.

Now we give the proof of Theorem \ref{pac}. For given $\alpha, \kappa, V$, let $$\mathcal{AR}=\{E \in \R \quad | \quad (\kappa \alpha, D_E^V) \quad \text{is almost reducible}\},$$
which is an open set since almost reducibility cocycles is an open set in $(\R \backslash \Q) \times C^{\omega}(\T,SL(2,\R))$  (Corollary 1.3 of \cite{avila2010almost}), i.e., 
$$
\mathcal{AR}= \cup_{j=1}^{J} I_j= \cup_{j=1}^{J} (a_j,b_j), $$
where $J\in \N$ or $J=\infty$.
 Note that the operators we consider are bounded, thus  we only need to prove purely absolutely continuous spectrum in  $\cup_{j=1}^{J}$ $(a_j,b_j) \cap (-M,M)$ for some $M>0$. Take any interval in $\cup_{j=1}^{J}  (a_j,b_j) \cap (-M,M),$ omiting $j$ and denote it by $(a, b)$.
To prove $H_{V,\alpha, \theta}$ has purely absolutely continuous spectrum in the bounded interval $(a,b)$, one only need to prove that  $H_{V,\alpha, \theta}$ has purely absolutely continuous spectrum in $$\mathcal{S}(\delta_0) = \Sigma(V) \cap [a+\delta_0, b-\delta_0 ]$$  for any sufficiently small $\delta_0>0$.

Now we give the full proof. First we need the following result, which states that for any $E\in \mathcal{S}(\delta_0)$, then  after a finite number (that is uniform with respect to $E\in \mathcal{S}(\delta_0)$) of conjugation steps, one can reduce the cocycle  to the perturbative regime.

\begin{lemma}\label{global-local}
For any $\epsilon_0>0,$  $\alpha\in \R\backslash \Q,$ there exist
$\bar{h}=\bar{h}(\alpha)>0$ and $\Gamma=\Gamma(\alpha,\epsilon_0)>0$ such that for any $
E\in  \mathcal{S}(\delta_0)$,
there exist $\Phi_{E}\in C^{\omega}_{\bar{h}}(\mathbb{T},PSL(2,\mathbb{R}))$ with $\|\Phi_{E}\|_{\bar{h}}<\Gamma$ such that
\begin{equation*}
    \Phi_{E}(\theta+\kappa \alpha)^{-1}D^{V}_{E}(\theta)\Phi_{E}(\theta)=R_{\Phi_{E}}e^{f_{E}(\theta)}
\end{equation*}
with $\left\|f_{E}\right\|_{\bar{h}}<\epsilon_0,$ $\left|\operatorname{deg} \Phi_{E}\right| \leq C|\ln \Gamma|$ for some constant $C=$
$C(V,\alpha)>0 .$ 
\end{lemma}
\begin{proof}
Similar proof first   appeared in Proposition 5.2 of \cite{LYZZ}, we give the proof just for completeness.   The crucial fact for this proposition is that we can choose $\bar{h}(\alpha)$ to be
independent of $E$ and $\epsilon_0,$ and choose $\epsilon_0$ to be independent of $E$.

For any $E \in  \mathcal{S}(\delta_0),$  there exists $h_{0}=h_{0}(E,\alpha)>0,$ such that
$$\Phi_{E}(\cdot+\kappa \alpha)^{-1} D_{E}^{V}(\cdot) \Phi_{E}(\cdot)=R_{\phi(E)}+F_{E}(\cdot),$$
with $\left\|F_{E}\right\|_{h_{0}}<\epsilon_0 / 2$ and $\left\|\Phi_{E}\right\|_{h_{0}}<\tilde{\Gamma}$ for some $\tilde{\Gamma}=\Gamma(\alpha, \epsilon_0, E)>0.$
Note that for any $E , E^{\prime} \in \mathcal{S}(\delta_0)$, 
$$||D_{E}^{V}-D_{E^{\prime}}^{V}||_{h_{0}}<C(V,h_{0})| E-E^{\prime}|.$$
Thus for any $E^{\prime} \in \mathbb{R},$ one has
$$\left\|\Phi_{E}(\cdot+\kappa \alpha)^{-1} D_{E^{\prime}}^{V}(\cdot) \Phi_{E}(\cdot)-R_{\phi(E)}\right\|_{h_{0}}<\frac{\epsilon_0}{2}+C\left|E-E^{\prime}\right|\left\|\Phi_{E}\right\|_{h_{0}}^{2}.$$
It follows that with the same $\Phi_{E},$ we have
$$
\left\|\Phi_{E}(\cdot+\kappa \alpha)^{-1} D_{E^{\prime}}^{V}(\cdot) \Phi_{E}(\cdot)-R_{\phi(E)}\right\|_{h_{0}}<\epsilon_0
$$
for any energy $E^{\prime }$ in a neighborhood $\mathcal{U}(E)$ of $E$. Since $ \mathcal{S}(\delta_0)$ is compact,
by compactness argument, we can select $h_{0}(E, \alpha), \Gamma(\alpha, \epsilon_0, E)>0$ to be independent of the energy $E$.
\end{proof}

Once having Lemma \ref{global-local}, one can apply the KAM scheme (Proposition \ref{onekam}) to get precise control of the growth of the cocycles in the resonant sets.  We inductively give the parameters, 
for any $\bar{h}>\tilde{h}>0$, $\gamma>0,\sigma>0$, define 
\begin{align*}
 h_0=\bar{h}, \qquad  \epsilon_0 \leq D_0(\frac{\gamma}{\kappa^{\sigma}},\sigma) (\frac{\bar{h}-\tilde{h}}{8})^{C_0\sigma},\end{align*}
 where $D_{0} = D_0(\gamma,\sigma)$ and $C_{0}$ are constant defined in Proposition \ref{onekam},  and define
$$
\epsilon_j= \epsilon_0^{2^j}, \quad h_j-h_{j+1}=\frac{\bar{h}-\frac{\bar{h}+\tilde{h}}{2}}{4^{j+1}}, \quad N_j=\frac{2|\ln\epsilon_j|}{h_j-h_{j+1}}.
$$
Then we have the following:
\begin{proposition}\label{local kam}
Let $\alpha\in DC( \gamma, \sigma)$. Then there exists $B_{j} \in C_{h_{j}}^{\omega}\left(\mathbb{T}, PSL(2, \mathbb{R})\right)$ with
$|\deg B_j |  \leq 2 N_{j-1}$,
such that
$$B_{j}^{-1}(\theta+\kappa \alpha) R_{\Phi_{E}}e^{f_{E}(\theta)} B_{j}(\theta)=A_{j}(E) e^{f_{j}(\theta)},$$
with  estimates  $
\left\|B_{j}\right\|_{0} \leq |\ln\epsilon_{j-1}|^{4\sigma}$, $\left\|f_{j}\right\|_{h_{j}} \leq \epsilon_{j}.$ 
Moreover,  for any $0<|n| \leq N_{j-1}$, denote 
\begin{equation*}
\begin{split}
\Lambda_{j}(n)=\left\{E\in \mathcal{S}(\delta_0):\|2 \rho(\kappa \alpha, A_{j-1}(E))- \langle n,\kappa \alpha \rangle\|_{\T}< \epsilon_{j-1}^{\frac{1}{15}}\right\}.
\end{split}
\end{equation*}
If 
 $E\in K_j:= \cup_{|n|=1}^{N_{j-1}}\Lambda_{j}(n)$, then $A_{j}(E)$ can be written as 
$$
A_{j}(E)=M^{-1} \exp \left(\begin{array}{cc}{i t_{j}} & {v_{j}} \\ {\bar{v}_{j}} & {-i t_{j}}\end{array}\right) M,
$$
where
$$M=\frac{1}{1+i}\begin{pmatrix}
1&-i\\1&i
\end{pmatrix},$$
with estimates
$$
|t_j|<\epsilon^{\frac{1}{16}}_{j-1},\ |v_{j}|<\epsilon_{j-1}^{\frac{15}{16}}.
$$
\end{proposition}

\begin{proof}
We prove Proposition \ref{local kam} by iteration. In the proof, we will omit its dependence on the energy $E$ for simplicity.
Assume that we have completed the $j$-th step and are at
the $(j+1)$-th KAM step, i.e. we already construct $B_{j} \in C_{h_j}^{\omega}\left(\mathbb{T}, P S L(2, \mathbb{R})\right)$
such that
$$
B_{j}^{-1}(\theta+\kappa \alpha) R_{\Phi_E}e^{f_E(\theta)} B_{j}(\theta)=A_{j} e^{f_{j}(\theta)},
$$
with estimates
$$||f_j||_{h_j}\le\epsilon_j,\ |\deg B_j|\leq 2 N_{j-1}.
$$

Note $\alpha\in DC( \gamma, \sigma)$ implies $\kappa \alpha\in DC( \frac{\gamma}{\kappa^{\sigma}}, \sigma)$. Then 
by our selection of $\epsilon_{0}$ (see also Remark \ref{uniform}),  one can check that
$$
\epsilon_{j} \leq \frac{D_{0}}{\left\|A_{j}\right\|^{C_{0}}}\left(h_{j}-h_{j+1}\right)^{C_{0}\sigma}.
$$
Indeed, $\epsilon_{j}$ on the left side of the inequality decays super-exponentially
with $j,$ while $\left(h_{j}-h_{j+1}\right)^{C_{0} \sigma}$ on the right side decays exponentially with $j$. Thus, Proposition \ref{onekam} can be applied iteratively,
consequently one can construct
$$
\bar{B}_{j+1} \in C_{h_{j+1}}^{\omega}\left(\mathbb{T}, PSL(2, \mathbb{R})\right), \ A_{j+1} \in S L(2, \mathbb{R}), \ f_{j+1} \in C^{\omega}_{h_{j+1}}\left(\mathbb{T}, sl(2, \mathbb{R})\right),
$$
such that
$$
\bar{B}_{j+1}^{-1}(\theta+\kappa \alpha) A_{j} e^{f_{j}(\theta)} \bar{B}_{j+1}(\theta)=A_{j+1} e^{f_{j+1}(\theta)}.
$$
Let $B_{j+1}(\theta)=B_{j}(\theta) \bar{B}_{j}(\theta),$ then 
$$
B_{j+1}^{-1}(\theta+\kappa \alpha) R_{\Phi_{E}}e^{f_{E}(\theta)} B_{j+1}(\theta)=A_{j+1} e^{f_{j+1}(\theta)}.
$$
Moreover, according to the resonance relation, we can distinguish the following two cases: 

\textbf{Non-resonant case:} If $E\notin\cup_{|n|=1}^{N_{j}}\Lambda_{j+1}(n)$,
i.e.  for any $n \in \mathbb{Z}$ with $0<|n| \leqslant N_{j},$ 
we have
$$
\left\|2\rho(\kappa \alpha,A_j)-\left<n, \kappa \alpha\right>\right\|_{\mathbb{R} / \mathbb{Z}} \geqslant \epsilon_{j}^{\frac{1}{15}},
$$
then by Proposition \ref{onekam}, we have 
\begin{equation}\label{conj2}
\|A_{j+1}-A_j\| \leq \epsilon_j, \quad \left\|\bar{B}_{j+1}-i d\right\|_{h_{j+1}} \leq \epsilon_{j}^{\frac{1}{2}}, \quad\left\|f_{j+1}\right\|_{h_{j+1}} \leq \epsilon_{j+1},
\end{equation}
which implies that 
$$ 
|\deg B_{j+1}|=|\deg B_{j}| \leq 2N_{j-1}\leq 2 N_j
.$$

\textbf{Resonant case:}  If $E\in\cup_{|n|=1}^{N_{j}}\Lambda_{j+1}(n)$,  i.e.  there exists $n$ with $0<\left|n\right| \leq N_j$ such that
$$\|2 \rho(\kappa \alpha, A_j)- \langle n,\kappa \alpha \rangle\|_{\mathbb{R} / \mathbb{Z}}< \epsilon_{j}^{\frac{1}{15}}.$$
By Proposition \ref{onekam}, we have 
$$
\left\|f_{j+1}\right\|_{h_{j+1}}\leq  \epsilon_j e^{-h_{j+1}\epsilon_j^{-\frac{1}{18\sigma}}}\leq \epsilon_{j+1},
$$
and the conjugacy satisfy 
\begin{equation}\label{conj1}
 \left\|\bar{B}_{j+1}\right\|_{0}< \frac{\kappa^{\sigma}}{\gamma} |n|^{\sigma}, \quad  \deg \bar{B}_j =n, 
\end{equation}
which implies that 
$$|\deg B_{j+1}|=|\deg B_{j}+\deg \bar{B}_{j}|\leq  2 N_{j-1}+N_j \leq 2N_j.$$
Moreover, we can write
$$
A_{j+1}=M^{-1} \exp \left(\begin{array}{cc}{i t_{j+1}} & {v_{j+1}} \\ {\bar{v}_{j+1}} & {-i t_{j+1}}\end{array}\right) M
$$
with estimates
$$
|t_{j+1}|<\epsilon_j^{\frac{1}{16}},\ |v_{j+1}|<\epsilon_{j}^{\frac{15}{16}}.
$$

Finally, we are left to prove 
$$||B_{j+1}||_0\le|\ln\epsilon_j|^{4\sigma}.$$
To estimate this, we need more detailed analysis on the resonances. Assume that there are at least two resonant steps, say the $(m_i+1)^{th}$ and $(m_{i+1}+1)^{th}$. At the $(m_{i+1}+1)^{th}$-step, the resonance condition implies
$$\|2 \rho(\kappa \alpha, A_{m_{i+1}})- \langle n_{m_{i+1}},\kappa \alpha \rangle\|_{\mathbb{R} / \mathbb{Z}}<  \epsilon_{m_{i+1}}^{\frac{1}{15}},$$
 hence by $\alpha\in DC(\gamma, \sigma)$,  we have 
 \begin{equation}\label{lower} |\rho(\kappa \alpha, A_{m_{i+1}}(E))|\geq  \frac{\gamma}{2 |\kappa n_{m_{i+1}}|^{\sigma}}  - \epsilon_{m_{i+1}}^{\frac{1}{15}}  \geq  \frac{\gamma}{3 |\kappa n_{m_{i+1}}|^{\sigma}}.\end{equation} On the other hand, according to Proposition \ref{onekam}, after the $(m_i+1)^{th}$-step, $|\rho(\kappa \alpha, A_{m_{i}}(E))|\leq \epsilon_{m_i}^{\frac{1}{16}}$.  Then \eqref{conj2} implies that 
 $|\rho(\kappa \alpha, A_{m_{i+1}}(E))| \leq 2 \epsilon_{m_i}^{\frac{1}{16}} ,$
 since by our selection, between $(m_i+1)^{th}$ and $(m_{i+1}+1)^{th}$ step, there are no resonant steps. Thus by \eqref{lower} and $|n_{m_i}| \leq  N_{m_i}=\frac{2|\ln\epsilon_{m_i}|}{h_{m_i}-h_{m_i+1}}$, we have
\begin{equation}\label{different resonances}
|n_{m_{i+1}}|\geq \frac{1}{2\kappa} \gamma^{\frac{1}{\sigma}} \epsilon_{m_i}^{-\frac{1}{16\sigma}}> |n_{m_i}|^2 .
\end{equation}
Assuming that there are $s$ resonant steps, associated with integers vectors
$$
n_{m_1},...,n_{m_s} \in\Z, \ \ 0<|n_{m_i}|\leq N_{m_i},\ \ i=1,...,s,
$$
By \eqref{conj2}, \eqref{conj1} and \eqref{different resonances}, we have 
\begin{equation*}
\begin{split}
||B_{j+1}||_0&\le2\prod_{i=1}^s||\bar B_{m_i}||_0\le2\prod_{i=1}^s\frac{\kappa^{\sigma}}{\gamma}|n_{m_i}|^\sigma\\&<2(\frac{\kappa^{\sigma}}{\gamma})^s|n_{m_s}|^{\sigma(1+\frac{1}{2}+\cdots\frac{1}{2^s})}\\&<2(\frac{\kappa^{\sigma}}{\gamma})^s|n_{m_s}|^{2\sigma}<|\ln\epsilon_{m_s}|^{4\sigma}<|\ln\epsilon_j|^{4\sigma}.
\end{split}
\end{equation*}
We thus finish the proof.

\end{proof}

\begin{remark}
As we noted in Proposition \ref{onekam}, in the resonant case, the new perturbation can be chosen as $\epsilon e^{-h_+\epsilon^{-\frac{1}{18\tau}}}$, which is much smaller than $\epsilon^2$. However, here we just choose $\epsilon_{j+1}=\epsilon_j^2$, otherwise if the perturbation $\epsilon_{j}(E)$ depends on $E$ (due to the fact the resonances depend on $E$),   one cann't give a good stratification of the energies in the spectrum. 
\end{remark}

In the construction, $K_j$ just means the cocycle $(\kappa \alpha, R_{\Phi_{E}}e^{f_{E}(\theta)})$ is resonant in the $j$-th KAM step. If $E\in K_j$, then we have the following characterization of its IDS and the growth of the cocycles in the resonant sets:

\begin{lemma}\quad\label{lemma4.7}
 Assume that $\alpha\in DC( \gamma, \sigma)$,  $E \in K_{j},$ then there exists $\tilde{n}_j \in \mathbb{Z}$ with $0<|\tilde{n}_j| <2N_{j-1}$ such that
\begin{equation}\label{nj}
\| \kappa N_{V}(E)+\langle \tilde{n}_j, \kappa \alpha\rangle\|_{\mathbb{R} / \mathbb{Z}} \leqslant 2 \epsilon_{j-1}^{\frac{1}{15}}.
\end{equation}
Moreover,  we have 
\begin{equation*}
\sup _{0 \leq s \leq \epsilon_{j-1}^{-\frac{1}{8}}}\|(D_{E}^{V})_s\|_{0} \leq 4\Gamma^2 |\ln\epsilon_{j-1}|^{8\sigma}.
\end{equation*}
\end{lemma}
\begin{proof}
First by Lemma \ref{global-local}  there exist $\Phi_{E}\in C^{\omega}(\mathbb{T},PSL(2,\mathbb{R}))$ with $\left|\deg \Phi_{E}\right| \leq C|\ln \Gamma|$  such that
\begin{equation*}
    \Phi_{E}(\theta+\kappa \alpha)^{-1}D^{V}_{E}(\theta)\Phi_{E}(\theta)=R_{\Phi_{E}}e^{f_{E}(\theta)}
\end{equation*}
Furthermore by Proposition \ref{local kam}, there exist  $B_{j-1} \in C_{h_{j-1}}^{\omega}\left(\mathbb{T}, PSL(2, \mathbb{R})\right)$  with $|\deg B_{j-1}|\le2N_{j-2}$   such that
\begin{equation*}
B_{j-1}(\theta+\kappa \alpha)^{-1} \Phi_{E}(\theta+\kappa \alpha)^{-1}D^{V}_{E}(\theta)\Phi_{E}(\theta) B_{j-1}(\theta)=A_{j-1} e^{f_{j-1}(\theta)},
\end{equation*}
and for any $E\in K_j$, we have
	\begin{equation}\label{rho3}
	||2 \rho\left(\kappa \alpha, A_{{j-1}}\right)-\left\langle n', \kappa \alpha\right\rangle||_{\mathbb{R} / \mathbb{Z}} \leqslant \epsilon_{j-1}^{\frac{1}{15}}
	,\end{equation}
	for some $n' \in \mathbb{Z}$ with $0<|n'|\le N_{j-1}$.
	
Thus, we deduce that
\begin{equation}\label{rho1}
2\rho(\kappa \alpha,D_E^{V})=\left<\deg B_{j-1}+\deg\Phi_E,\kappa \alpha\right>+2\rho(\kappa \alpha,A_{j-1}e^{f_{j-1}(\theta)}).\end{equation}
Note that $||f_{j-1}||_0\le\epsilon_{j-1}$, we have
\begin{equation}\label{rho2}|\rho(\kappa \alpha,A_{j-1}e^{f_{j-1}(\theta)})-\rho(\kappa \alpha,A_{j-1})|\le c\epsilon_{j-1}^{\frac{1}{2}}.\end{equation}
Combining \eqref{rho3}, \eqref{rho1} and \eqref{rho2} suggest that
$$||2\rho(\kappa \alpha,D_E^{V})-\left<\deg B_{j-1}+\deg\Phi_E,\kappa \alpha\right>-\left<n',\kappa \alpha\right>||_{\mathbb{R}/\mathbb{Z}}\le2\epsilon_{j-1}^{\frac{1}{15}}.$$
Let $\tilde{n}_j=\deg B_{j-1}+n' + \deg \Phi_E$, then 
\begin{equation*}
\|2 \rho(\kappa \alpha,D_E^{V})-\langle \tilde{n}_j, \kappa \alpha\rangle\|_{\mathbb{R} / \mathbb{Z}} \leqslant 2 \epsilon_{j-1}^{\frac{1}{15}},
\end{equation*}
with estimate
 $$|\tilde{n}_j|\le2N_{j-2}+N_{j-1} + C|\ln \Gamma| \le2N_{j-1}. $$
Next we observe that 
$$\kappa \rho(T_\alpha,S_{E}^{V})=\rho(\kappa \alpha,D_E^{V}) \ mod \ \mathbb{Z} ,$$
consequently by the fact that $N_{V}(E)=1-2\rho(T_\alpha,S_{E}^{V}),$  we obtain \eqref{nj}.

On the other hand, note if $B, D$ are small $sl(2, \mathbb{R})$ matrices,
then there exists $E\in sl(2, \mathbb{R})$ such that
\begin{equation*}
e^{B}e^{D}=e^{B+D+E},
\end{equation*}
where $E$ is a sum of terms at least 2 orders in $B,D.$ Thus by Proposition \ref{local kam}, there exist  $B_{j} \in C_{h_{j}}^{\omega}\left(\mathbb{T}, PSL(2, \mathbb{R})\right)$ with $\|B_j\|_0\le|\ln\epsilon_{j-1}|^{4\sigma}$ such that
\begin{equation*}
B_{j}(\theta+\kappa \alpha)^{-1} \Phi_{E}(\theta+\kappa \alpha)^{-1}D^{V}_{E}(\theta)\Phi_{E}(\theta) B_{j}(\theta)= R_{t_j} e^{ \tilde{f}_{j}(\theta) }
\end{equation*}
with estimates $|t_j|\le\epsilon_{j-1}^{\frac{1}{16}}$,  $\| \tilde{f}_j\|_0\le\epsilon_{j-1}^{\frac{1}{8}}$.  This imply that 
$$\sup_{0\le s\le \epsilon_{j-1}^{-\frac{1}{8}}}||(D_E^{V})_s||_0\le 4 \|\Phi_{E}\|^2_{\bar{h}} \|B_j\|_0^2\le 4\Gamma^2|\ln\epsilon_{j-1}|^{8\sigma}.$$
\end{proof}
%
%

Next we study the regularity of $N_{V}$. The observation here is that while $(T_\alpha, S_{E}^{V})$ is not an analytic quasi-periodic cocycle, its  iterate $(\kappa \alpha, D_{E}^{V})$ indeed defines an analytic quasi-periodic  cocycle, then we one can pass the $\frac{1}{2}$-H\"older continuity of $(\kappa \alpha, D_{E}^{V})$ to $(T_\alpha, S_{E}^{V})$, the rest proof is standard.

\begin{lemma}\label{idsholder}
 Assume that $\alpha\in DC( \gamma, \sigma)$,  then the integrated density of states $N_{V}$ is $\frac{1}{2}$-H\"older continuous on $\mathcal{AR}$.
\end{lemma}
\begin{proof}
The assumption $\alpha\in DC( \gamma, \sigma)$ implies $\kappa \alpha\in DC( \frac{\gamma}{\kappa^{\sigma}}, \sigma)$. Thus for any $E\in \mathcal{S}(\delta_0)$, we apply the following result:
 
\begin{lemma}\label{ids}
Let $\tilde{\alpha} \in DC$,  if  $(\tilde{\alpha}, A)$ is  analytically almost reducible,  then for any
continuous map $B : \mathbb{T} \rightarrow SL(2, \mathbb{C}),$ we have
$$
|L(\tilde{\alpha}, A)-L(\tilde{\alpha}, B)| \leqslant C_{0}\|B-A\|_{0}^{\frac{1}{2}}
$$
where $C_{0}$ is a constant depending on $\alpha$.
\end{lemma}\begin{proof}
This is essentially contained in the Corollary 4.6 of \cite{AJ1},  see also Proposition 7.1 of \cite{LYZZ}.
\end{proof}
  Consequently,   by Lemma \ref{ids},   we have 
$ | L(\kappa \alpha, D_{E+i\epsilon}^{V}) - L(\kappa \alpha, D_{E}^{V}) | \leq C \epsilon^{\frac{1}{2}},$
which implies that
$|L(T_\alpha, S_{E+i\epsilon}^{V})- L(T_\alpha, S_{E}^{V})| \leq \frac{C}{\kappa} \epsilon^{\frac{1}{2}}.$
On the other hand,   Thouless formula state that 
$$
L(\alpha, S_{E}^{V})=\int \ln \left|E-E^{\prime}\right| d N_{V}\left(E^{\prime}\right),
$$
then for every $\epsilon>0$, we have
\begin{equation*}
\begin{split} |L(T_\alpha, S_{E+i\epsilon}^{V})- L(T_\alpha, S_{E}^{V})|&=\frac{1}{2} \int \ln \left(1+\frac{\epsilon^{2}}{\left(E-E^{\prime}\right)^{2}}\right) d N_{V} \left(E^{\prime}\right)\\&\ge\frac{1}{2} \int_{E-\epsilon}^{E+\epsilon} \ln \left(1+\frac{\epsilon^{2}}{\left(E-E^{\prime}\right)^{2}}\right) d N_{V} \left(E^{\prime}\right)\\ & \ge \frac{\ln2}{2} \left(N_{V}(E+\epsilon)-N_{V}(E-\epsilon)\right), 
\end{split}
\end{equation*}
which gives 
$$N_{V}(E+\epsilon)-N_{V}(E-\epsilon) \le \frac{C}{\ln2} \epsilon^{\frac{1}{2}}.$$
Since $N_{V}$ is locally constant in the complement of $\Sigma(V)$, this means precisely that  $N_{V}$ is $\frac{1}{2}$-H\"older continuous.
\end{proof}

As a consequence of Lemma \ref{idsholder}, we can show $N_{V}$ has a lower bound estimate:

\begin{lemma}\label{lemma4.10}
For any $\delta_{0}>0$ which is small enough, if $E \in\mathcal{S}(\delta_0)$, then for sufficiently small $\epsilon>0$, $$N_{V}(E+\epsilon)-N_{V}(E-\epsilon)\geq c(\delta_0)\epsilon^{\frac{3}{2}},$$
where $c(\delta_0) > 0$ is a small universal constant.
\end{lemma}

\begin{proof} 
The proof is first developed by Avila \cite{avila2010almost} in the quasi-periodic case, which can be generalized to the general case almost without change, we sketch the proof here just for completeness. Let $\delta=c \epsilon^{3 / 2}$. For any  $E \in \mathcal{S}(\delta_0),$ we have $L(T_\alpha,S_E^{V})=0,$ then by Thouless formula we have
$$
L(T_\alpha,S_{E+i\delta}^{V})=\int \frac{1}{2} \ln (1+\frac{\delta^{2}}{\left|E-E^{\prime}\right|^{2}})\ d N_{V}\left(E^{\prime}\right).
$$
We split the integral into four parts: $I_{1}=\int_{\left|E-E^{\prime}\right|\ge\frac{\delta_{0}}{2}}$, $ I_{2}=\int_{\epsilon\le\left|E-E^{\prime}\right|<\frac{\delta_{0}}{2}}$, $I_{3}=\int_{\epsilon^{4}\le\left|E-E^{\prime}\right|<\epsilon}$ and $I_{4}=\int_{\left|E-E^{\prime}\right|<\epsilon^{4}}$.

For sufficiently small $\epsilon>0$, by Lemma \ref{idsholder} we have $I_1 <\frac{2c^2\epsilon^3}{\delta_0^2},$
and
\begin{equation*}
\begin{split}
I_{4}&=\sum_{k \geq 4} \int_{\epsilon^{k}>\left|E-E^{\prime}\right|\ge\epsilon^{k+1}}  \frac{1}{2}\ln (1+\frac{\delta^{2}}{\left|E-E^{\prime}\right|^{2}})\ d N_{V}\left(E^{\prime}\right) \\
&\leq \frac{1}{2}\sum_{k \geq 4} \epsilon^{\frac{k}{2}} \ln (1+c^{2} \epsilon^{1-2 k})\leq \epsilon^{\frac{7}{4}}.
\end{split}
\end{equation*}
We also have the estimate
$$
\begin{aligned} I_{2} & \leq \sum_{k=0}^{m} \int_{e^{-k-1}\le|E-E'|<e^{-k}} \frac{1}{2}\ln (1+\frac{\delta^{2}}{\left|E-E^{\prime}\right|^{2}})\ d N_{V}\left(E^{\prime}\right)\\&\le\sum_{k=0}^{m}\frac{1}{2}e^{-\frac{k}{2}}\delta^2e^{2k+2}\leq Cc^2\delta,\end{aligned}
$$
with $m=[-\ln \epsilon]$. 
It follows that
$I_3\ge L(T_\alpha,S_{E+i\delta}^{V})-\frac{1}{40}\delta-Cc^2\delta.$
It is well known that $L(T_\alpha,S_{E+i \delta}^{V}) \geq\delta/10$, for
$0<\delta<1 $ \cite{DPSB}. Since the constant $c$ above is consistent with our choice of $\delta$, we can shrink it such that $I_3\ge\frac{1}{20}\delta$. Since $I_{3} \leq C(N_{V}(E+\epsilon)-N_{V}(E-\epsilon)) \ln \epsilon^{-1},$
the result follows.
\end{proof}

%
%
%

\textbf{Proof  of Theorem \ref{pac}}
%

Let $\mathcal{B}$ be the set of $E\in  \mathcal{S}(\delta_0) $ such that the Schr\"odinger cocycle $(T_{\alpha}, S_{E}^{V})$ is bounded, which equals to  the set $(\kappa \alpha,D_E^{V})$ is bounded. 
 By  Theorem \ref{subordinnacy}, it is enough to prove that
$\mu_{V,\theta}( \mathcal{S}(\delta_0)\backslash\mathcal{B})=0$ for any $\theta\in\R$.

 Let $\mathcal{R}$ be the set of $E\in  \mathcal{S}(\delta_0) $ such that $(\kappa \alpha,D_E^{V})$ is reducible, then $\mathcal{R}\backslash\mathcal{B}$ only contains $E$ for which $(\kappa \alpha,D_E^{V})$ is analytically reducible to a constant parabolic cocycle.  Recall that  for any $E\in  \mathcal{S}(\delta_0)$, by well-known result of Eliasson \cite{E92}, if 
 $\rho(\kappa \alpha, D_{E}^{V})$ is rational or  Diophantine w.r.t $\kappa \alpha$,  then $(\kappa \alpha,D_E^{V})$ is reducible. It follows that $\mathcal{R}\backslash\mathcal{B}$
  is countable: indeed for any such $E$ there exists  $m\in\mathbb{Z}$ such that 
 $2\rho(\kappa \alpha,D_E^{V})=\left<m,\alpha\right>\mod\mathbb{Z}.$  Moreover, if $E\in\mathcal{R}$, then any non-zero solution of $H_{V,\alpha,\theta}u=Eu$, satisfies  $\inf_{n\in\mathbb{Z}}|u_{\kappa n}|^{2}+|u_{\kappa n+1}|^{2}>0$, so there are no eigenvalues in $\mathcal{R}$ and $\mu_{V,\theta}(\mathcal{R}\backslash\mathcal{B})=0$.  Therefore, it is enough to show that for sufficiently small $\delta_{0}>0$, 
$\mu_{V,\theta}( \mathcal{S}(\delta_0)\backslash\mathcal{R})=0.$
Note that $\mathcal{S}(\delta_0)\backslash\mathcal{R}\subset\limsup K_{m}$, by Borel-Cantelli Lemma, we only need to prove $\sum_m \mu_{V,\theta}(\overline{K}_{m})<\infty$.

Let $J_{m}(E)$  be an open $\epsilon_{m-1}^{\frac{2}{45}}$  neighborhood of $E\in K_m$. By Lemma \ref{growth-cocycle} and Lemma \ref{lemma4.7}, we have
\begin{equation*}
\begin{split}
\mu_{V,\theta}(J_m(E))&\le \sup_{0\le s\le\epsilon_{m-1}^{-\frac{2}{45}}}||(D_E^{V})_s||_0^2|J_m(E)|\\&\le\sup_{0\le s\le\epsilon_{m-1}^{-\frac{1}{8}}}||(D_E^{V})_s||_0^2|J_m(E)|\le C|\ln\epsilon_{m-1}|^{16\sigma}\epsilon_{m-1}^{\frac{2}{45}}.
\end{split}
\end{equation*}
Let $\cup_{l=0}^rJ_m(E_l)$ be a finite subcover of $\overline{K}_m$. By refining this subcover, we can assume that every $E\in\mathbb{R}$ is contained in at most two different $J_m(E_l)$. 

On the other hand, by Lemma \ref{lemma4.7}, if $E\in K_m$, then
$$|| \kappa N_{V}(E)+\left<n, \kappa \alpha\right>||_{\mathbb{R/Z}}\le 2\epsilon_{m-1}^{\frac{1}{15}},$$
for some $|n|<2N_{m-1}$.
This shows that $\kappa N_{V}({K}_{m})$ can be covered by $2N_{m-1}$  intervals $I_{s}$ of length  $2 \epsilon_{m-1}^{\frac{1}{15}}$.
By Lemma \ref{lemma4.10},
$$ \kappa N_{V}(J_m(E))\ge c|J_m(E)|^{\frac{3}{2}},$$
thus  by our selection $|I_{s}|\leq \frac{1}{c}| \kappa N_{V}(J_{m}(E))|$   for any $s$
and $E\in{K}_{m}$, there are at most $2([\frac{1}{c}]+1)+4$ intervals $J_m(E_l)$ such that $\kappa N_{V}(J_m(E_l))$  intersects $I_{s}$.
We conclude that there are at most $2(2([\frac{1}{c}]+1)+4)N_{m-1}$
intervals $J_m(E_l)$ to cover $K_m$. Then 
\begin{equation*}
\mu_{V,\theta}(\overline{K}_{m})\leq \sum_{j=0}^{r}\mu_{V}(J_{m}(E_{j}))\le CN_{m-1}|\ln\epsilon_{m-1}|^{16\sigma}\epsilon_{m-1}^{\frac{2}{45}}< \epsilon_{m-1}^{\frac{1}{45}},
\end{equation*}
which gives $\sum_m \mu_{V,\theta}(\overline{K}_{m})<\infty$.

\section{Anderson Localization}

In this section, we will prove Anderson localization for GAA model and quasi-periodic mosaic model. We will  fix $V=V_1$ (the GAA model) or $V_2$ (quasi-periodic mosaic model) in this section, and for the quasi-periodic mosaic model, we will just consider $\kappa=2$, since the general case follows similarly.

\begin{theorem}\label{al-thm}
Suppose that $\alpha\in DC(\gamma,\sigma)$. Then  $H_{V,\alpha,\theta}$ has Anderson Localization in the set 
$$ \mathcal{P} := \Sigma(V) \cap \{E \in \R \quad | \quad L(\alpha, S_{E}^{V})>0\}$$ for every $\theta\in \Theta$, where 
$$ \Theta= \cup_{\eta>0} \Theta(\eta) := \cup_{\eta>0}  \{\theta \in \R | \quad  \|2\theta- k\alpha\| \geq \frac{\eta}{|k|^{\sigma}} \quad \forall k\neq 0 \}.$$
\end{theorem}

We start with the basic setup going back to \cite{jitomirskaya1999metal}. 
  We will use the notation $G_{[n_{1},n_{2}]}(n,m)$ for the Green's function $(H_{V,\alpha,\theta}-E)^{-1}(n,m)$ of the operator $H_{V,\alpha,\theta}$ restricted to the interval $[n_{1},n_{2}]$ with zero boundary conditions at $n_{1}-1$ and $n_{2}+1$. To simplify the notations, we replace $L(\alpha,S_{E}^{V})$ by $L(E)$, 
  the $V,\alpha$-dependence of various quantities will be omitted in some cases.

Denote by $M_{k}(\theta)$ the $k-$step transfer-matrix of $H_{V,\alpha,\theta}u=Eu$, and denote 
\begin{equation*}
P_{k}(\theta)=\det[(H_{V,\alpha,\theta}-E)|_{[0,k-1]}],\quad Q_{k}(\theta)=\det[(H_{V,\alpha,\theta}-E)|_{[1,k]}],
\end{equation*}
then the $k-$step transfer-matrix can be written as
\begin{equation*}
M_{k}(\theta)=(-1)^{k}\begin{pmatrix}P_{k}(\theta)&Q_{k-1}(\theta)\\-P_{k-1}(\theta)&-Q_{k-2}(\theta)\end{pmatrix}.
\end{equation*}
By Kingman's Subadditive Ergodic Theorem,  Lyapunov exponent satisfies
\begin{equation}\label{equation_26}
L(E)=\inf_{n\ge1}\frac{1}{n}\int_{\mathbb{T}}\ln\|M_{n}(\theta)\|d\theta=\lim_{n\to\infty}\frac{1}{n}\ln\|M_{n}(\theta)\|,
\end{equation}
for almost every $\theta\in\mathbb{T}$.
Moreover, if we  recall the following Furman's result:

\begin{theorem}\label{Furman}\cite{F}
	Suppose $(X,T)$ is uniquely ergodic. If $f_{n}:X\to\mathbb{R}$ are continuous and satisfy  $f_{n+m}(\theta)\le f_{n}(\theta)+f_{m}(T^{n}\theta)$, then
	$$\limsup_{n\to\infty}\frac{1}{n}f_{n}(\theta)\le \inf_{n\ge1}\frac{1}{n}\mathbb{E}(f_{n})$$
	for every $\theta\in X$ and uniformly on $X$.
\end{theorem}

Then we have uniform growth of the transfer matrix:

\begin{lemma}\label{lemma5.2}
For every $E\in\mathbb{R}$ and $\epsilon>0$, there exists $k_1(E,\epsilon)$ such that $$\|M_{k}(\theta)\|<e^{(L(E)+\epsilon)k}$$ for every $k>k_1(E,\epsilon)$ and every $\theta\in\mathbb{T}$.
\end{lemma}
\begin{proof}
If $V=V_1,$  then the base dynamics is $(\mathbb{T},R_{\alpha})$ which is uniquely ergodic. If $V=V_2$,  the base dynamics is $(\mathbb{T}\times\mathbb{Z}_{\kappa},T_{\alpha})$, which is also uniquely ergodic (Theorem 9.1 of \cite{mane2012ergodic}). Apply Theorem \ref{Furman} to $f_{n}(\theta)=\ln\|M_{n}(\theta)\|$, we get desired results. \end{proof}

When the Lyapunov exponent is positive, Lemma \ref{lemma5.2} implies that  some of the entries must be exponentially large. These entries in turn appear in a description of the Green's function of the operator restricted to a finite interval. Namely, by Cramer's Rule, if we denote \begin{equation*}
\Delta_{m,n}(\theta)=\det[(H_{V,\alpha,\theta}-E)|_{[m,n]}],
\end{equation*} 
then for $n_{1},n_{2}=n_{1}+k-1$, and $n\in[n_{1},n_{2}]$,
\begin{equation}\label{equation_28}
\begin{split}
|G_{[n_{1},n_{2}]}(n_{1},n)|&=|\frac{\Delta_{n+1,n_{2}}(\theta)}{\Delta_{n_{1},n_{2}}(\theta)}|,\\
|G_{[n_{1},n_{2}]}(n,n_{2})|&=|\frac{\Delta_{n_{1},n-1}(\theta)}{\Delta_{n_{1},n_{2}}(\theta)}|.
\end{split}
\end{equation}
A useful definition about Green's function is the following:
\begin{definition}\cite{jitomirskaya1999metal}
Fix $E\in\mathbb{R}$ and $\xi\in\mathbb{R}$. A point $n\in\mathbb{Z}$ will be called $(\xi,k)$-regular if there exists an interval $[n_{1},n_{2}]$, $n_{2}=n_{1}+k-1$, containing $n$, such that
\begin{equation*}
|G_{[n_{1},n_{2}]}(n,n_{i})|<e^{-\xi|n-n_{i}|},   \quad  and    \ |n-n_{i}|\ge\frac{1}{7} k;\ i=1,2.
\end{equation*}
Otherwise, $n$ will be called $(\xi,k)$-singular.
\end{definition}

It is well known  that  any formal solution $u$ of the $H_{V,\alpha,\theta}=Eu$ at a point $n\in[n_{1},n_{2}]$ can be reconstructed from the boundary values via
\begin{equation}\label{boundary5.5}
u(n)=-G_{[n_{1},n_{2}]}(n,n_{1})u(n_{1}-1)-G_{[n_{1},n_{2}]}(n,n_{2})u(n_{2}+1).
\end{equation}
This implies that if $u_E$ is a generalized eigenfunction, then every point $n\in\mathbb{Z}$ with $u_E(n)\neq0$ is $(\xi,k)$-singular for $k$ sufficiently large: $k>k_{2}(E,\xi,\theta,n)$.
In the following,  we just just assume $u_E(0)\neq0$ (otherwise replace $u_E(0)$ by $u_E(1)$). Then Theorem \ref{al-thm} will follow from the next result:

\begin{proposition}\label{lemma5.6}
Assume that $\alpha\in DC$, $\theta \in \Theta$,  $L(E)>0$.  Then for every $\epsilon>0$, for any $|y|> y(\alpha,\theta,E,\epsilon)$ sufficiently large, there exists $k>\frac{5}{16}|y|,$ such that $y$ is $(L(E)-\epsilon,k)-$regular.
\end{proposition}

\medskip

\textbf{Proof of Theorem \ref{al-thm}.}

It is well known  that if every generalized eigenfunction  of $H_{V,\alpha,\theta}$ decays exponentially, then the operator $H_{V,\alpha,\theta}$ displays Anderson localization. Let $E\in \mathcal{P}$ be a generalized eigenvalue of $H_{V,\alpha,\theta}$, and denote the corresponding generalized eigenfunction by $u_{E}$. Let $\epsilon$ small enough,
	by  \eqref{boundary5.5} and Proposition \ref{lemma5.6}, if $|y|>y(\alpha,\theta,E,\epsilon)$
	the point $y$ is $(L(E)-\epsilon,k)$-regular for some $k>\frac{5}{16}|y|$. Thus, there exists an interval $[n_{1},n_{2}]$ of length $k$ containing $y$ such that 
	$\frac{1}{7}k\le |y-n_{i}|\le\frac{6}{7}k,$
	and
	$$|G_{[n_{1},n_{2}]}(y,n_{i})|<e^{-(L(E)-\epsilon)|y-n_{i}|}, \quad i=1,2.$$
Using \eqref{boundary5.5}, we obtain that
	\begin{equation*}
	\begin{split}
	|u_{E}(y)|\le2C(E)(2|y|+1)e^{-\frac{L(E)-\epsilon}{7}k} \le e^{-\frac{L(E)-\epsilon}{24}|y|}.
	\end{split}
	\end{equation*}
	This implies exponential decay of the eigenfunction if $\epsilon$  is chosen small enough.
\qed

\subsection{Anderson localization for the GAA model.}

For the GAA model, the basic observation is that  $Q_{k}(\theta)=P_{k}(\theta+\alpha)$, then all the elements of $M_k(\theta)$ can be expressed by $Q_k(\theta)$.  The key observation is the following:

\begin{lemma}\label{lemma5.5} 
	We have
	$$Q_{k}(\theta)=\frac{R_k(\cos2\pi(\theta+\frac{k+1}{2}\alpha))}{\prod\limits_{j=1}^{k}(1-\tau \cos2\pi  (\theta+j\alpha))},$$
	where $R_k(\cdot)$ is a polynomial of degree k.
\end{lemma}

\begin{proof}
First notice that 
	$$g_{k}(\theta):=Q_{k}(\theta)\prod\limits_{j=1}^{k}(1-\tau \cos2\pi (\theta+j\alpha))=\det\begin{pmatrix}c_{1}&d_{2}& & \\d_{1}&\ddots&\ddots& \\ &\ddots&\ddots&d_{k}\\ & &d_{k-1}&c_{k}\end{pmatrix}$$
	where $c_{n}=(\tau E+2\lambda)\cos2\pi(\theta+n\alpha)-E$, $d_{n}=1-\tau \cos2\pi(\theta+n\alpha)$,  then it is a trigonometric polynomial with degree less than $k$.

 On the other hand, since $\cos 2\pi \theta$ is an even function, denote $U$ the change of basis $\delta_{j}\mapsto\delta_{k+1-j}$, then
	$$U^{-1}H_{V_{1},\alpha,\theta-\frac{k+1}{2}\alpha}|_{[1,k]}U=H_{V_{1},\alpha,-\theta-\frac{k+1}{2}\alpha}|_{[1,k]},$$
which implies that 
	$$Q_{k}(\theta-\frac{k+1}{2}\alpha)=Q_{k}(-\theta-\frac{k+1}{2}\alpha).$$
Due to the fact
	$$\prod\limits_{j=1}^{k}(1-\tau \cos2\pi (\theta-\frac{k+1}{2}\alpha+j\alpha))=\prod\limits_{j=1}^{k}(1-\tau \cos2\pi (-\theta-\frac{k+1}{2}\alpha+j\alpha)),$$
then we  have
	\begin{equation*}\label{equation_29}
	g_{k}(\theta-\frac{k+1}{2}\alpha)=g_{k}(-\theta-\frac{k+1}{2}\alpha).
	\end{equation*}
	Therefore, we obtain
	$$g_{k}(\theta)=\sum\limits_{j=0}^{k}\widetilde{a}_{j}\cos(2\pi j(\theta+\frac{k+1}{2}\alpha)),$$
	since the linear span of $\{1,\cos(2\pi x),$ $\cos(2\pi 2x),$ $\dots,\cos(2\pi kx)\}$ is equal to that of $\{1,\cos(2\pi x),$ $\cos^{2}(2\pi x),$ $\dots,\cos^{k}(2\pi x)\}$, consequently we have
	\begin{equation*}
	Q_{k}(\theta)=\frac{\sum\limits_{j=0}^{k}\widetilde{a}_{j}\cos2\pi j(\theta+\frac{k+1}{2}\alpha)}{\prod\limits_{j=1}^{k}(1-\tau \cos2\pi (\theta+j\alpha))}\\
	=\frac{\sum\limits_{j=0}^{k}a_{j}(\cos2\pi(\theta+\frac{k+1}{2}\alpha))^{j}}{\prod\limits_{j=1}^{k }(1-\tau \cos2\pi (\theta+j\alpha))}.
	\end{equation*}
\end{proof}

By Lemma \ref{lemma5.5}, if we denote 
$$M_{k,r}=\{x\in\mathbb{T}:|R_k(\cos2\pi x)|\le e^{(k+1)(r+\ln\frac{1+\sqrt{1-\tau^2}}{2})}\},$$
then we have the following:

\begin{lemma}\label{lemma5.9.}
	Suppose $y\in\mathbb{Z}$ is $(L(E)-\epsilon,k)$-singular. Then  
	for every $j$ satisfying $y-\frac{5}{6}k\le j\le y-\frac{1}{6}k$, we have $\theta+(j+\frac{k+1}{2})\alpha$ belongs to $M_{k,L(E)-\frac{\epsilon}{8}}$ for $k>k_3(E,\epsilon)$.
\end{lemma}

\begin{proof}
	Since $y$ is $(L(E)-\epsilon,k)$-singular, without loss of generality, assume that for every interval
	$[j+1,j+k]$ of length $k$ containing $y$ with $y-\frac{5}{6}k\le j\le y-\frac{1}{6}k$, then $|y-j-1|>\frac{1}{7}k$ and $|j+k-y|>\frac{1}{7}k$, we have that
	$$|G_{[j+1,j+k]}(y,j+k)|\ge e^{-(L(E)-\epsilon)|y-j-k|}.$$
Using \eqref{equation_28}, we have
	\begin{equation*}
	\begin{split}
	|G_{[j+1,j+k]}(y,j+k)| = |\frac{Q_{y-j-1}(\theta+j\alpha)}{Q_{k}(\theta+j\alpha)}|\ge e^{-(L(E)-\epsilon)|y-j-k|},
	\end{split}
	\end{equation*}
by	 Lemma \ref{lemma5.2}, we obtain
	$$|Q_{y-j-1}(\theta+j\alpha)|\le e^{|y-j-1|(L(E)+\frac{\epsilon}{90})},\quad \text{for} \quad  k>k_1(E,\frac{\epsilon}{90}).$$
which implies that 
	\begin{equation*}
	\begin{split}
	|Q_{k}(\theta+j\alpha)|&\le e^{kL(E)+(\frac{6}{7}k \times \frac{1}{90}-\frac{1}{7}k)\epsilon}=e^{k(L(E)-\frac{2\epsilon}{15})}.
	\end{split}
	\end{equation*}
	
	On the other hand,  by  Jensen's formula  and uniquely ergodicity,  
	\begin{eqnarray*}
		\lim_{k\to\infty}  \frac{1}{k}\sum\limits_{j=1}^{k}\ln(1-\tau \cos(2\pi(\theta+j\alpha))) 
= \ln\frac{1+\sqrt{1-\tau^2}}{2}, \quad  \forall \theta \in\R, 
	\end{eqnarray*}
	thus there exists $k_3(E,\epsilon)>k_1(E,\frac{\epsilon}{90})$ such that if $k>k_3(E,\epsilon)$, then 
	$$\prod\limits_{j=1}^{k}(1-\tau \cos2\pi  (\theta+j\alpha))<e^{k(\ln\frac{1+\sqrt{1-\tau^2}}{2}+\frac{\epsilon}{120})},  \quad \forall \theta \in\R $$
		Consequently, by Lemma \ref{lemma5.5}, we have 
	\begin{equation*}
	\begin{split}
	R_k(\cos2\pi(\theta+(j+\frac{k+1}{2})\alpha))&\le e^{k(L(E)-\frac{2\epsilon}{15}+\ln\frac{1+\sqrt{1-\tau^2}}{2}+\frac{\epsilon}{120})}\\&\le e^{(k+1)(L(E)+\ln\frac{1+\sqrt{1-\tau^2}}{2}-\frac{\epsilon}{8})}.
	\end{split}
	\end{equation*}
which just means $\theta+(j+\frac{k+1}{2})\alpha\in M_{k,L(E)-\frac{\epsilon}{8}}$.
	\end{proof}

 On the other hand, we may write the polynomial $R_k(x)$ in  Lagrange interpolation form 
\begin{equation}\label{lagrange}
	|R_{k}(x)|=|\sum\limits_{j=0}^{k}R_{k}(\cos2\pi\theta_{j})\frac{\prod\nolimits_{l\neq j}(x-\cos2\pi\theta_{l})}{\prod\nolimits_{l\neq j}(\cos2\pi\theta_{j}-\cos2\pi\theta_{l})}|
	\end{equation}
and introduce the following useful definition:

\begin{definition}
	We say that the set $\{\theta_0,\cdots,\theta_{k}\}$ is $\epsilon-$uniform if
	$$\max_{x\in[-1,1]}\max_{i=0,\cdots,k}\prod_{j=0,j\neq i}^{k}\frac{|x-\cos2\pi\theta_j|}{|\cos2\pi\theta_i-\cos2\pi\theta_j|}<e^{k\epsilon}.$$
\end{definition}

\begin{lemma}\label{lemma5.11.}
	Let $0<\epsilon'<\epsilon$, $L(E)>0$. If $\theta_0,\cdots,\theta_k\in M_{k,L(E)-\epsilon}$, then $\{\theta_0,\cdots,\theta_{k}\}$ is not $\epsilon'-$uniform for $k>k_4(\epsilon,\epsilon')$.
\end{lemma}
\begin{proof}
	Otherwise, using \eqref{lagrange} we  get
	\begin{equation*}
	|R_{k}(x)|\le (k+1)e^{(k+1)(L(E)-\epsilon+\ln\frac{1+\sqrt{1-\tau^2}}{2}+\epsilon')},
	\end{equation*}
	for any $x\in[-1,1]$. On the other hand, if $k\geq k_4(\epsilon,\epsilon')0$ is large enough,  then  $$\prod\limits_{j=1}^{k}(1-\tau \cos2\pi  (\theta+j\alpha))>e^{k(\ln\frac{1+\sqrt{1-\tau^2}}{2}-\frac{\epsilon-\epsilon'}{2})},$$
which implies that  $|Q_k(\theta)|\le e^{k(L(E)-\frac{\epsilon-\epsilon'}{3})}$ for all $\theta\in\mathbb{T}$.  However, by Herman's subharmonic function argument, $\int_{\mathbb{R/Z}}\ln|Q_{k}(x)|dx\ge kL(E)$, this is a contradiction. 
\end{proof}

We consider two points 0 and $y$, without loss of generality, assume $y>0$. 
let $k=2\lfloor \frac{3}{8}y \rfloor+1$,
$n_{1}=-\lfloor\frac{3}{4}k\rfloor,n_2=y-\lfloor\frac{3}{4}k\rfloor$. Then we can construct the following sequence:
\begin{equation*}
\theta_{j}=\left\{
\begin{matrix}
\theta+(n_{1}+\frac{k-1}{2}+j)\alpha,&j=0,1,\dots,\lfloor\frac{k+1}{2}\rfloor-1,\\
\theta+(n_{2}+\frac{k-1}{2}+j-\lfloor\frac{k+1}{2}\rfloor)\alpha,&j=\lfloor\frac{k+1}{2}\rfloor,\lfloor\frac{k+1}{2}\rfloor+1,\dots,k.
\end{matrix}\right .
\end{equation*}
These points $\theta_{0},\theta_{1},\dots,\theta_{k}$ are distinct and satisfy the following:

\begin{lemma}\label{corollary5.11}
Suppose that $\alpha\in DC$, $\theta\in\Theta$. Then for any $\epsilon>0$, there exists $k_5(\alpha,\theta,\epsilon)>0$, such that for $k>k_5(\alpha,\theta,\epsilon)$,  
the above constructed sequence $\{\theta_j\}_{j=0}^k$ is $\epsilon-$uniform.
\end{lemma}
\begin{proof}
This is essentially Lemma 7 of \cite{jitomirskaya1999metal}.
\end{proof}

\smallskip
\textbf{Proof of Proposition \ref{lemma5.6}: GAA case.}

By Lemma \ref{lemma5.11.} and Lemma \ref{corollary5.11}, we know that $\{\theta_j\}_{j=0}^k$ can not be inside $M_{k,L(E)-\frac{\epsilon}{8}}$ at the same time for sufficiently large $k$. Since $u_E$ is a generalized function satisfying $u_E(0)\neq0$, $0$ is $(L(E)-\epsilon,k)-$singular for sufficiently large $k$.  Applying Lemma \ref{lemma5.9.}, one obtains $$\{\theta_j\}_{j=0}^{\lfloor\frac{k+1}{2}\rfloor-1}\subset M_{k,L(E)-\frac{\epsilon}{8}}.$$
Assume $y$ is $(L(E)-\epsilon,k)$-singular, then we also have
$$\{\theta_j\}_{j=\lfloor\frac{k+1}{2}\rfloor}^{k}\subset M_{k,L(E)-\frac{\epsilon}{8}}.$$ 
Thus $\{\theta_j\}_{j=0}^k\subset M_{k,L(E)-\frac{\epsilon}{8}}$, this contradiction means $y$ must be $(L(E)-\epsilon,k)-$regular for $y>y(\alpha,\theta,E,\epsilon)$. Notice that $k=2\lfloor\frac{3}{8}|y|\rfloor+1>\frac{5}{16}|y|$, we thus finish  the proof of GAA case.

\subsection{Anderson localization for the mosaic model.}

Note in the GAA case, one of the basic observation is that  the elements of $M_k(\theta)$ can be expressed by $Q_k(\theta)$. In the quasi-periodic mosaic case, the transfer matrix reads as
\begin{equation*}
	M_{2k}(\theta)=\begin{pmatrix}P_{2k}(\theta)&Q_{2k-1}(\theta)\\-P_{2k-1}(\theta)&-Q_{2k-2}(\theta)\end{pmatrix},
	\end{equation*}
the key observation is that elements of $M_{2k}(\theta)$ can be written as linear combination of $Q_{2k-1}(\theta)$ (possibly with different $k$ and different $\theta$):

\begin{lemma}\label{mk}
 We have 
	\begin{equation*}
	\begin{split}
	EQ_{2k-2}(\theta)&=- Q_{2k-1}(\theta)-Q_{2k-3}(\theta),\\
	E P_{2k}(\theta)&=-Q_{2k+1}(\theta-2\alpha)-Q_{2k-1}(\theta),\\
	E^2 P_{2k-1}(\theta)&=Q_{2k+1}(\theta-2\alpha)+Q_{2k-1}(\theta)+Q_{2k-1}(\theta-2\alpha)+Q_{2k-3}(\theta).
	\end{split}
	\end{equation*}
\end{lemma}

\begin{proof}
Note that $V_{2}(\theta,2n+1)=0$ and $V_{2}(\theta,n+2)=V_{2}(\theta+2\alpha,n)$. Then if we expand  the determinant $\det[(H_{V,\alpha,\theta}-E)|_{[0,2k]}]$  by the last column, we have
	\begin{equation*}
	\begin{split}
	P_{2k}(\theta)&=-EP_{2k-1}(\theta)-P_{2k-2}(\theta),\\
	EQ_{2k-2}(\theta)&=- Q_{2k-1}(\theta)-Q_{2k-3}(\theta).
	\end{split}
	\end{equation*}
Meanwhile, if we expand  the determinant $\det[(H_{V,\alpha,\theta}-E)|_{[0,2k]}$ by the first column, we have
	\begin{equation*}
	Q_{2k-1}(\theta)=-EP_{2k-2}(\theta+2\alpha)-Q_{2k-3}(\theta+2\alpha).
	\end{equation*}
	which implies that 
	\begin{eqnarray}
	E P_{2k}(\theta)&=&-Q_{2k+1}(\theta-2\alpha)-Q_{2k-1}(\theta) \nonumber \\
	E^2 P_{2k-1}(\theta)&=&Q_{2k+1}(\theta-2\alpha)+Q_{2k-1}(\theta)+Q_{2k-1}(\theta-2\alpha)+Q_{2k-3}(\theta).\nonumber 
	\end{eqnarray}
We thus finish the proof. 	
\end{proof}

Similar to Lemma  \ref{lemma5.5}, we have the following:
\begin{lemma}\label{lemma5.9}
For every $k\in2\mathbb{N}+1$,  there exists a polynomial $\widetilde R_{\frac{k-1}{2}}$ of degree $\frac{k-1}{2}$ such that
\begin{equation*}
Q_{k}(\theta)=\widetilde R_{\frac{k-1}{2}}(\cos2\pi(\theta+\frac{k+1}{2}\alpha).
\end{equation*}

\end{lemma}
\begin{proof}
	Since  $\cos 2\pi \theta$  is  an even function, it follow that the changes of basis $\delta_{j}\mapsto\delta_{2k+2-j}$ transforms
	\begin{equation*}
	H_{V_{2},\alpha,\theta-(k+1)\alpha}|_{1,2k+1} \qquad  \text{into} \qquad  H_{V_{2},\alpha,-\theta-(k+1)\alpha}|_{1,2k+1}.
	\end{equation*}
which implies that 
	\begin{equation*}\label{equation_34}
	Q_{2k+1}(\theta-(k+1)\alpha)=Q_{2k+1}(-\theta-(k+1)\alpha).
	\end{equation*}
	 The rest  proof is similar to  Lemma \ref{lemma5.5}, we thus omit the details.  
\end{proof}

By Lemma \ref{lemma5.9},  if we  denote the set $$\widetilde M_{k,r}=\{x\in\mathbb{T}:|\widetilde R_{\frac{k-1}{2}}(\cos2\pi x)|\le e^{(k+1)r}\},$$
then we have the following:

\begin{lemma}\label{lemma5.13.}
Suppose $y\in\mathbb{Z}$ is $(L(E)-\epsilon,k)-$singular, $k\in2\mathbb{N}+1$. Then for every $j\in\mathbb{Z}$ satisfying $y-\frac{5}{6}k+\frac{k+1}{2}\le 2j \le y-\frac{1}{6}k+\frac{k+1}{2}$, we have $\theta+2j\alpha$ belongs to $\widetilde M_{k,L(E)-\frac{\epsilon}{8}}$ for $k>k_1(E,\frac{\epsilon}{48})$.
\end{lemma}

\begin{proof}
The proof is similar to Lemma \ref{lemma5.9.}, we omit the details. 
\end{proof}

%
\begin{lemma}\label{noun}
Let $0<\epsilon'<\epsilon$, $k\in2\mathbb{N}+1$, $L(E)>0$. If $\theta_0,\cdots,\theta_{\frac{k-1}{2}}\in\widetilde M_{k,L(E)-\epsilon}$, then $\{\theta_0,\cdots,$ $\theta_{\frac{k-1}{2}}\}$ is not $\epsilon'-$uniform for $k>k_6(\epsilon,\epsilon')$.
\end{lemma}
\begin{proof}
Otherwise, using Lagrange interpolation form \eqref{lagrange}, we get  $|\widetilde R_{\frac{k-1}{2}}(x)|<e^{k(L(E)-\frac{\epsilon-\epsilon'}{2})}$ for all $x\in[-1,1]$ when $k>k_6(\epsilon,\epsilon')$, which implies \begin{equation*}\label{qk}
|Q_k(\theta)|<e^{k(L(E)-\frac{\epsilon-\epsilon'}{2})}, \qquad \forall \theta\in \R.
\end{equation*} 
On the other hand, Lemma \ref{mk} imply that $$||M_{2n}(\theta)\|\le C\max\{|Q_{2n+1}(\theta-2\alpha)|,|Q_{2n-1}(\theta)|,|Q_{2n-1}(\theta-2\alpha)|,|Q_{2n-3}(\theta)|\},$$
for some constant $C=C(\lambda)$, since by Corollary \ref{moscor}, we have $2+\lambda\ge|E|>\frac{1}{\lambda}$.
However, this contradicts to \eqref{equation_26}
for sufficiently large $n$. We thus finish the proof.
\end{proof}

Assume that $(q_{n})_{n}$ is the sequence of denominators of the best rational approximations of $2\alpha$. 
Select $n$ such that $q_n\le\frac{y}{8}<q_{n+1}$ and let $s$ be the largest positive integer satisfying $sq_n\le\frac{y}{8}$. Set $I_1,I_2\subset\mathbb{Z}$ as follows
$$I_1=[0,sq_n-1]\text{ and }I_2=[1+\lfloor\frac{y}{2}\rfloor-sq_n,\lfloor\frac{y}{2}\rfloor+sq_n].$$
\begin{lemma}\label{guji}
Let $\theta_j=\theta+2j\alpha$, then for any $\epsilon>0$, the set $\{\theta_j\}_{j\in I_1\cup I_2}$ is $\epsilon-$uniform if $y>y(\alpha,\theta,\epsilon)$.
\end{lemma}
\begin{proof}
 Take $x=\cos2\pi a$. Now it suffices to estimate
$$\sum_{j\in I_1\cup I_2,j\neq i}(\ln|\cos2\pi a-\cos2\pi\theta_j|-\ln|\cos2\pi\theta_i-\cos2\pi\theta_j|)=\sum_1-\sum_2.$$
Then Lemma \ref{ten} reduces this problem to estimating the minimal terms.

First we estimate $\sum_1$:
\begin{equation*}
\begin{split}
\sum_1 &=\sum_{j\in I_1\cup I_2,j\neq i}\ln|\sin\pi(a+\theta_j)|+\sum_{j\in I_1\cup I_2,j\neq i}\ln|\sin\pi(a-\theta_j)|+(3sq_n-1)\ln2\\
&=\sum_{1,+}+\sum_{1,-}+(3sq_n-1)\ln2,
\end{split}
\end{equation*}
we cut $\sum_{1,+}$ or $\sum_{1,-}$ into $3s$ sums and then apply Lemma \ref{ten}, we get that for some absolute constant $C_1$:
$$\sum_1\le-3sq_n\ln2+C_1s\ln q_n.$$

Next, we estimate $\sum_2$ as follows:
\begin{equation*}
\begin{split}
\sum_2&=\sum_{j\in I_1\cup I_2,j\neq i}\ln|\sin\pi(2\theta+(i+j)2\alpha)|\\&\quad+\sum_{j\in I_1\cup I_2,j\neq i}\ln|\sin\pi(i-j)2\alpha|+(3sq_n-1)\ln2\\&=\sum_{2,+}+\sum_{2,-}+(3sq_n-1)\ln2.
\end{split}
\end{equation*}
For any $0<|j|<q_{n+1}$, since $\alpha\in DC(\gamma,\tau)$  we have $$\|j2\alpha\|_{\mathbb{R/Z}}\ge\|q_n2\alpha\|_{\mathbb{R/Z}}\ge \frac{\gamma}{ (2q_n)^{\sigma}}.$$ Therefore we obtain 
$$\max\{\ln|\sin x|,\ln|\sin(x+\pi j2\alpha)|\}\ge2\ln\gamma-2\sigma\ln2q_n\quad\text{for}\ y>y_1(\alpha).$$
This means in any interval of length $sq_n$, there can be at most one term which is less than $2\ln\gamma-2\sigma\ln2q_n$. Then there can be at most 3 such terms in total.

For the part $\sum_{2,-}$, since $$\|(i-j)2\alpha\|_{\mathbb{R/Z}}\ge \frac{\gamma}{2^{\sigma} |i-j|^{\sigma}}\ge \frac{\gamma}{(18sq_n)^{\sigma}},$$ these 3 smallest terms must be bounded by $\ln\gamma-\sigma\ln18sq_n$ from below. Hence by Lemma \ref{ten}, we have 
\begin{equation}\label{s2}\sum_{2,-}\ge-3sq_n\ln2+3\ln\gamma-3\sigma\ln18sq_n-C_2s\ln q_n,\end{equation}
for $y>y_2(\alpha)$ and some absolute constant $C_2$.
For the part $\sum_{2,+}$, since $\theta\in \Theta$, then $$\|2\theta+(i+j)2\alpha\|_{\mathbb{R/Z}}\ge \frac{\eta}{|i+j|^{\sigma}}\ge \frac{\eta}{(18sq_n)^{\sigma}},$$ these 3 smallest terms must be greater than $\ln\eta-\sigma\ln18sq_n$.  Therefore combining with \eqref{s2}, we have $$\sum_2\ge-3sq_n\ln2+3\ln\gamma-3\sigma\ln18sq_n+3\ln\eta-3\sigma\ln18sq_n-(C_2+C_3)s\ln q_n,$$
consequently,  for any $\epsilon>0$ if $y>y(\alpha,\theta,\epsilon)$,
$\sum_1-\sum_2\leq 6 \epsilon sq_n,$
i.e.  the set $\{\theta_j\}_{j\in I_1\cup I_2}$ is $\epsilon-$uniform.
\end{proof}

\medskip
\textbf{Proof of Proposition  \ref{lemma5.6}:  Quasi-periodic mosaic case:}

Combining Lemma \ref{noun} and Lemma \ref{guji}, we know that when $y$ is sufficiently large, $\{\theta_j\}_{j\in I_1\cup I_2}$ can not be inside the set $\widetilde M_{6sq_n-1,L(E)-\frac{\epsilon}{8}}$ at the same time. Therefore 0 and $y$ can not be $(L(E)-\epsilon,6sq_n-1)$-singular at the same time by Lemma \ref{lemma5.13.}. However 0 is $(L(E)-\epsilon,6sq_n-1)-$singular given $y$ large enough. Therefore
$$\{\theta_j\}_{j\in I_1}\subset\widetilde M_{6sq_n-1,L(E)-\frac{\epsilon}{8}}.$$
Thus $y$ must be $(L(E)-\epsilon,6sq_n-1)-$regular for $y>y(\alpha,\theta,E,\epsilon)$. Notice that $6sq_n-1>6/16y-1>\frac{5}{16}y$, thus we complete the proof.

\section{Proof of Main results}

\textbf{Proof of Theorem \ref{thm-gaa}:}  By Corollary \ref{theorem3.9},  Theorem  \ref{thm-gaa} (1) and the first statement  of Theorem  \ref{thm-gaa} (3)  follow from  Theorem \ref{arc} and Theorem \ref{pac}. 
Theorem  \ref{thm-gaa} (2) and the second statement  of Theorem  \ref{thm-gaa} (3)  follow from Theorem \ref{al-thm}.\\

\textbf{Proof of Theorem \ref{thm-mosaic}:}
The proof is same as Theorem \ref{thm-gaa}, one only needs to replace Corollary \ref{theorem3.9} by Corollary \ref{moscor}.\\

\textbf{Proof of Theorem \ref{thm-mosaic-k3}:}
The proof is same as Theorem \ref{thm-gaa}, one only needs to replace Corollary \ref{theorem3.9} by Corollary \ref{moscor-k3}.\\

Theorem \ref{thm-mosaic}  and Theorem \ref{thm-mosaic-k3} covers the quasi-periodic mosaic model $\kappa=2$ and $\kappa=3$, for the general $\kappa$, recall that 
 \begin{eqnarray*}
		a_{\kappa}(E)=
		\frac{1}{\sqrt{E^2-4}}\left( (\frac{E+ \sqrt{E^2-4}}{2})^{\kappa}- (\frac{E- \sqrt{E^2-4}}{2})^{\kappa} \right),\ \
	\end{eqnarray*}
and we have the following

\begin{theorem}\label{thm-mosaic-2}
	For any $\lambda \neq 0$,  $\alpha\in DC$, $\kappa\in \Z^+$, then  $|\lambda a_{\kappa}(E)|=1$ are the MEs. More precisely, 
	\begin{enumerate}
	 \item  $H_{V_2,\alpha,\theta}$ has purely absolutely continuous spectrum for every $\theta$  in  \begin{equation} \label{ac}\Sigma(V_2)\cap \{E \in \R | |\lambda a_{\kappa}(E)|<1 \}.\end{equation}
	\item If  $$\Sigma(V_2)\cap \{E \in \R | |\lambda a_{\kappa}(E)|>1 \} \neq \emptyset,$$  then  $H_{V_2,\alpha,\theta}$ has Anderson localization  in this set  for almost every $\theta$.
	\end{enumerate}
\end{theorem}   
\begin{proof} The proof is same as Theorem \ref{thm-mosaic}, one only needs to replace  Corollary \ref{moscor} by
Lemma \ref{lemma3.11} and Lemma  \ref{zeroenergy}.\end{proof}

\textbf{Proof of Corollary \ref{tb}:}
By Aubry duality, we only to need consider  its dual operator $H_{V_3,\alpha,\theta}$. By Corollary \ref{corollary3.11},   Theorem \ref{arc} and Theorem \ref{pac},  $H_{V_3,\alpha,\theta}$ has purely absolutely continuous spectrum in $sgn(\lambda)E<2\cosh p-\frac{2}{|\lambda|}$ for every $\theta$. By Corollary \ref{corollary3.11} and Theorem \ref{al-thm},  $H_{V_3,\alpha,\theta}$ has Anderson localization in    $sgn(\lambda)E>2\cosh p-\frac{2}{|\lambda|}$ for a.e. $\theta$. By Aubry duality \cite{BJ02,GJLS}, and the fact that  $ \Sigma(\widehat H_{V_3,\alpha,\theta} ) =\frac{\lambda}{2}\Sigma(V_3),$
ME of \eqref{model_3} has the form  $sgn(\lambda)E=2\cosh p-\frac{2}{|\lambda|}$, which is just  $E+1= 2|\lambda|\cosh p$.\\

\textbf{Proof of Corollary \ref{model_peaky}:}
By Corollary \ref{corollary3.13},  then Corollary \ref{model_peaky} follows from Theorem \ref{thm-gaa}.

\appendix

\section{A quantitative almost reducibility result}
The following quantitative almost reducibility result from \cite{CCYZ,LYZZ} is the basis of our
proof. 
\begin{proposition}\label{onekam}
Let $\alpha\in DC( \gamma, \sigma)$. Suppose that $A \in SL(2, \R)$,
$f \in C_{h}^{\omega}\left(\mathbb{T}, sl(2, \R)\right).$ Then for any $h_{+}<h,$ there exists numerical constant
$C_{0},$ and constant $D_{0}=D_{0}\left( \gamma, \sigma\right)$ such that if
$$
\|f\|_{h} \leq \epsilon \leq \frac{D_{0}}{\|A\|^{C_{0}}}\left(\min \left\{1, \frac{1}{h}\right\}\left(h-h_{+}\right)\right)^{C_{0} \sigma},
$$
then there exist $B \in C_{h_{+}}^{\omega}\left(2\T, PSL(2, \R)\right), $ such that
$$
B^{-1}(\theta+\alpha) A e^{f(\theta)} B(\theta)=A_{+} e^{f_{+}(\theta)}
$$
More precisely, let $spec(A)=\left\{e^{2 \pi i \xi}, e^{-2 \pi i \xi}\right\}, N=\frac{2}{h-h_{+}}|\ln \epsilon|$,  then we can distinguish two cases: 
\begin{itemize}
\item (Non-resonant case) if for any $n \in \mathbb{Z}$ with $0<|n| \leq N,$ we have
$$
\|2 \xi-\left<n, \alpha\right>\|_{\mathbb{R} / \mathbb{Z}} \geq \epsilon^{\frac{1}{15}}
$$
then
$$
\|B-i d\|_{h_{+}} \leq \epsilon^{\frac{1}{2}}, \quad\left\|f_{+}\right\|_{h_{+}} \leq \epsilon^{2}
$$
Moreover, $\left\|A_{+}-A\right\|<2 \epsilon$.
\item(Resonant case) if there exists $n_{*}$ with $0<\left|n_{*}\right| \leq N$ such that
$$
\left\|2 \xi-\left<n_{*}, \alpha\right>\right\|_{\mathbb{R} / \mathbb{Z}}<\epsilon^{\frac{1}{15}}
$$
then we have
$$
\|B\|_{h_{+}} \leq \frac{1}{\gamma}|n_*|^{ \frac{\tau}{2}} \epsilon^{-\frac{h_{+}}{h-h_{+}}}, \quad\left\|B\right\|_{0} \leq \frac{1}{\gamma}|n_*|^{ \frac{\tau}{2}}, \quad \|f_{+}\|_{h_+}< \epsilon e^{-h_+\epsilon^{-\frac{1}{18\tau}}}.$$
Moreover, $\deg B=n_{*}$, letting $
M=\frac{1}{1+i}\begin{pmatrix} 1 & -i\\ 1 & i \end{pmatrix}
$, then the constant $A_{+}$ can be written as
$$
A_{+}=M^{-1} \exp \left(\begin{array}{cc}{i t_{+}} & {\nu_{+}} \\ {\bar{\nu}_{+}} & {-i t_{+}}\end{array}\right) M
$$
with estimates   $|t_{+}| \leq \epsilon^{\frac{1}{16}}$,   $\left|\nu_{+}\right| \leq \epsilon^{\frac{15}{16}} e^{-2 \pi\left|n_{*}\right| h}$.
\end{itemize}

\end{proposition}

\begin{remark}\label{uniform}
Assume that  $A$ varies in some compact subset  of $SL(2,\R)$. Then $\epsilon$ can be taken uniform with respect to $A$.
\end{remark}

\section*{Acknowledgements}
The authors would like to thank D. Damanik, R. Krikorian and S. Jitomirskaya for useful discussions.   X. Xia, J. You and  Q.Zhou were partially  supported by
 National Key R\&D Program of China (2020YFA0713300)  and  Nankai Zhide Foundation.  Y. Wang is supported by the NSFC grant (12061031).  J. You was also partially supported by NSFC grant (11871286).  Z. Zheng  acknowledges financial  supports  of NSFC grant (12031020, 11671382), CAS Key Project of Frontier Sciences (No. QYZDJ-SSW-JSC003), the Key Lab. of Random Complex Structures and Data  Sciences CAS and National Center for Mathematics and Interdisciplinary Sciences CAS. Q.Zhou was  supported by NSFC grant (12071232), the Science Fund for Distinguished Young Scholars of Tianjin (No. 19JCJQJC61300).

\end{document}